\documentclass[reqno]{amsart}
\usepackage{amssymb}

\usepackage[colorlinks=true]{hyperref}
\hypersetup{urlcolor=blue, citecolor=red}

\newtheorem{Theorem}{Theorem}[section]

\newtheorem{Lemma}[Theorem]{Lemma}
\newtheorem{Proposition}[Theorem]{Proposition}
\newtheorem{Remark}[Theorem]{Remark}

\numberwithin{equation}{section}

\def\C {\mathbb C}

\def\R {\mathbb R}

\newcommand{\<}{\langle}
\renewcommand{\>}{\rangle}
\renewcommand{\(}{\left(}
\renewcommand{\)}{\right)}

\newcommand{\id}{\operatorname{Id}}
\newcommand{\inclusion}{\hookrightarrow}
\newcommand{\lesim}{\lesssim}

\newcommand{\Div}{\operatorname{Div}}

\newcommand{\p}{\partial}
\newcommand{\Vol}{\operatorname{Vol}}

\renewcommand{\Re}{\operatorname{Re}}

\begin{document}
\title[Direct and Inverse problems for nonlinear Maxwell equations]{Direct and Inverse problems for the nonlinear time-harmonic Maxwell equations in~Kerr-type~media}

\author[Yernat M. Assylbekov]{Yernat M. Assylbekov }
\address{Department of Computational and Applied Mathematics, Rice University, Houston, TX 77005, USA}
\email{yernat.assylbekov@gmail.com}

\author[Ting Zhou]{Ting Zhou}
\address{Department of Mathematics, Northeastern University, Boston, MA 02115, USA}
\email{t.zhou@neu.edu}

\maketitle

\begin{abstract}
{ In the current paper we consider an inverse boundary value problem of electromagnetism in a nonlinear Kerr medium. We show the unique determination of the electromagnetic material parameters and the nonlinear susceptibility parameters of the medium by making electromagnetic measurements on the boundary. We are interested in the case of the time-harmonic Maxwell equations.}
\end{abstract}


\section{Introduction}\label{section::introduction}


Let $(M,g)$ be a compact $3$-dimensional Riemannian manifold with smooth boundary and let $\mathcal E(\cdot,t)$ and $\mathcal H(\cdot,t)$ be the time-dependent $1$-forms on $M$ representing electric and magnetic fields. By $d$ and $*$ we denote the exterior derivative and the Hodge star operator on $(M,g)$, respectively. Consider the time-dependent Maxwell equations on the manifold, with no scalar charge density and with no current density
\begin{equation}\label{eqn::time-dependent Maxwell}
\begin{cases}
\p_t\mathcal B+*d\mathcal E=0,\\
\p_t\mathcal D-*d\mathcal H=0,\\
*d*\mathcal D=0,\\
*d*\mathcal B=0,
\end{cases}
\end{equation}
where $\mathcal D$ and $\mathcal B$ are $1$-forms representing electric displacement and magnetic induction
$$
\mathcal D=\varepsilon\mathcal E+\mathcal P_{NL}(\mathcal E),\quad \mathcal B=\mu \mathcal H+{ \mathcal M_{NL}}(\mathcal H),
$$
with $\mathcal P_{NL}$ and ${ \mathcal M_{NL}}$ being the nonlinear polarization and nonlinear magnetization, respectively. The (time-independent) functions $\varepsilon$ and $\mu$ on $M$, with positive real parts, represent the material parameters (permettivity and permeability, respectively).
\smallskip

The electric and magnetic fields $\mathcal E$ and $\mathcal H$ are said to be time-harmonic with frequency $\omega>0$ if
$$
\mathcal E(x,t)=E(x)e^{-i\omega t}+\overline{E(x)}e^{i\omega t},\quad \mathcal H(x,t)=H(x)e^{-i\omega t}+\overline{H(x)}e^{i\omega t},
$$
for some complex $1$-forms $E$ and $H$ on $M$. Then the time-averages of the intensities of $\mathcal E$ and of $\mathcal H$ are
$$
\frac{1}{T}\int_0^T|\mathcal E(x,t)|^2_g\,dt=2|E(x)|_g^2,\quad \frac{1}{T}\int_0^T|\mathcal H(x,t)|^2_g\,dt=2|H(x)|_g^2,
$$
where $T=2\pi/\omega$. In a medium with high intensity electric field, the nonlinear polarization is of the form
$$
\mathcal P_{NL}(x,\mathcal E(x,t))=\chi_e(x,|E|_g^2)\mathcal E(x,t),
$$
where $\chi_e$ is the scalar susceptibility depending only on the time-average of the intensity of $\mathcal E$. One of the most common nonlinear polarizations appearing in physics and engineering is the Kerr nonlinearity
\begin{equation}\label{eqn::form of chi_e}
\chi_e(x,|E|_{g}^2)=a(x)|E|_{g}^{2}.
\end{equation}
The reader is refereed to \cite{Nie1993,Stuart1991} for this and other examples of electric nonlinear phenomenas.\smallskip

We also assume that the nonlinear magnetization has the similar form
$$
{ \mathcal M_{NL}}(x,\mathcal H(x,t))=\chi_m(x,|H|_g^2)\mathcal H(x,t),
$$
where $\chi_m$ is the scalar susceptibilities depending only on the time-average of the intensity of $\mathcal H$. Such nonlinear magnetizations appear in the study of metamaterials built by combining an array of wires and split-ring resonators embedded into a Kerr-type dielectric~\cite{ZharovShadrivovKivshar2003}. These metamaterials have complicated form of nonlinear magnetization. However, if the intensity $|H|_g^2$ is sufficiently small, relatively to the resonant frequency, the nonlinear magnetization can be assumed to be of the Kerr-type \cite{KourakisShukla2005,LazaridesTsironis2005,WenXiangDaiTangSuFan2007}
\begin{equation}\label{eqn::form of chi_m}
\chi_m(x,|H|_{g}^2)=b(x)|H|_{g}^{2}.
\end{equation}
This assumption has successful numerical implementation \cite{LazaridesTsironis2005}.\smallskip

For the time-harmonic $\mathcal E$ and $\mathcal H$, the time-dependent Maxwell's system \eqref{eqn::time-dependent Maxwell} reduces to the nonlinear time-harmonic Maxwell equations for complex $1$-forms $E$ and $H$, with a fixed frequency $\omega>0$, will be
\begin{equation}\label{eqn::Maxwell}
\begin{cases}
* dE=i\omega\mu H+i\omega b |H|_{g}^{2}H,\\
* dH=-i\omega \varepsilon E-i\omega a |E|_{g}^{2}E.
\end{cases}
\end{equation}
The complex functions $\mu$ and $\varepsilon$ represent the material parameters (permettivity and permeability, respectively).

\subsection{Direct problem}
First we consider the boundary value problem for the nonlinear Maxwell equations~\eqref{eqn::Maxwell}. We suppose that $\varepsilon,\mu\in C^1(M)$ are complex functions with positive real parts and $a,b\in { C^1}(M)$.\smallskip

The boundary conditions are expressed in terms of \emph{tangential trace}. The latter is defined on $m$-forms by
$$
\mathbf{t}:C^\infty\Omega^{m}(M)\to C^\infty\Omega^{m}(\p M),\quad \mathbf{t}(w)=\imath^*(w),\quad w\in C^\infty\Omega^{m}(M),
$$
where $\imath:\p M\inclusion M$ is the canonical inclusion. Then $\mathbf t$ has its extension to a bounded operator $W^{1,p}\Omega^m(M)\to W^{1-1/p,p}\Omega^m(\p M)$ for $p>1$. Here and in what follows, $W^{1,p}\Omega^m(M)$ and $W^{1-1/p,p}\Omega^m(\p M)$ are standard Sobolev spaces of $m$-forms on $M$ and $\p M$, respectively.\smallskip

To describe the boundary conditions, we introduce the spaces
\begin{align*}
W^{1,p}_{\Div}(M)&=\{u\in W^{1,p}\Omega^1(M):\Div(\mathbf t(u))\in W^{1-1/p,p}\Omega^1(\p M)\},\\
TW^{1-1/p,p}_{\Div}(\p M)&=\{f\in W^{1-1/p,p}\Omega^1(\p M):\Div(f)\in W^{1-1/p,p}\Omega^1(\p M)\},
\end{align*}
where $\Div$ is the surface divergence on $\p M$; see Section~\ref{section::surface divergence} for the exact definition of $\Div$. These spaces are Banach spaces with norms
\begin{align*}
\|u\|_{W^{1,p}_{\Div}(M)}&=\|u\|_{W^{1,p}\Omega^1(M)}+\|\Div(\mathbf t(u))\|_{W^{1-1/p,p}(\p M)},\\
\|u\|_{TW^{1-1/p,p}_{\Div}(\p M)}&=\|f\|_{W^{1-1/p,p}(\p M)}+\|\Div(f)\|_{W^{1-1/p,p}(\p M)}.
\end{align*}
It is not difficult to see that $\mathbf t(W^{1,p}_{\Div}(M))=TW^{1-1/p,p}_{\Div}(\p M)$.\smallskip

Our first main result is the following theorem on well-posedness of the nonlinear Maxwell equations~\eqref{eqn::Maxwell} with prescribed small $\mathbf t(E)$ on $\p M$.

\begin{Theorem}\label{thm::main result 1}
Let $(M,g)$ be a compact $3$-dimensional Riemannian manifold with smooth boundary and let $3<p\leq 6$. Suppose that $\varepsilon,\mu\in C^1(M)$ are complex functions with positive real parts and $a,b\in C^1(M)$. For every $\omega\in \C$, outside a discrete set $\Sigma\subset\C$ of resonant frequencies, there is $\epsilon>0$ such that for all $f\in TW^{1-1/p,p}_{\Div}(\p M)$ with $\|f\|_{TW^{1-1/p,p}_{\Div}(\p M)}<\epsilon$, the Maxwell's equation \eqref{eqn::Maxwell} has a unique solution $(E,H)\in W^{1,p}_{\Div}(M)\times W^{1,p}_{\Div}(M)$ satisfying $\mathbf{t}(E)=f$ and
$$
\|E\|_{W^{1,p}_{\Div}(M)}+\|H\|_{W^{1,p}_{\Div}(M)}\le C\|f\|_{TW^{1-1/p,p}_{\Div}(\p M)},
$$
for some constant $C>0$ independent of $f$.
\end{Theorem}


\subsection{Inverse problem}

For $\omega>0$ with $\omega\notin \Sigma$, we define the \emph{admittance map} $\Lambda_{\varepsilon,\mu,a,b}^\omega$ as
$$
\Lambda_{\varepsilon,\mu,a,b}^\omega(f)=\mathbf t(H),\quad f\in TW^{1-1/p,p}_{\Div}(\p M),\quad \|f\|_{TW^{1-1/p,p}_{\Div}(\p M)}<\epsilon
$$
where $(E,H)\in W^{1,p}_{\Div}(M)\times W^{1,p}_{\Div}(M)$ is the unique solution of the system \eqref{eqn::Maxwell} with $\mathbf t(E)=f$, guaranteed by Theorem~\ref{thm::main result 1}. Moreover, the estimate provided in Theorem~\ref{thm::main result 1} implies that the admittance map satisfy
$$
\|\Lambda_{\varepsilon,\mu,a,b}^\omega(f)\|_{TW^{1-1/p,p}_{\Div}(\p M)}\le C\|f\|_{TW^{1-1/p,p}_{\Div}(\p M)}<C\epsilon.
$$
The inverse problem is to determine $\varepsilon,\mu,a$ and $b$ from the knowledge of the admittance map $\Lambda_{\varepsilon,\mu,a,b}^\omega$.\smallskip

To state our second main result, let us introduce the notion of admissible manifolds.\medskip

\noindent{\bf Definition. }A compact Riemannian manifold $(M,g)$ with smooth boundary of dimension $n\ge 3$, is said to be \emph{admissible} if $(M,g)\subset\subset\R\times (M_0,g_0)$, $g=c(e\oplus g_0)$ where $c>0$ smooth function on $M$, $e$ is the Euclidean metric and $(M_0,g_0)$ is a simple $(n-1)$-dimensional manifold. We say that a compact manifold $(M_0,g_0)$ with boundary is \emph{simple}, if $\p M_0$ is strictly convex, and for any point $x\in M_0$ the exponential map $\exp_x$ is a diffeomorphism from its maximal domain in $T_x M_0$ onto~$M_0$.
\medskip

Compact submanifolds of Euclidean space, the sphere minus a point and of hyperbolic space are all examples of admissible manifolds.
\smallskip

The notion of admissible manifolds were introduced by Dos Santos Ferreira, Kenig, Salo and Uhlmann~\cite{DosSantosFerreiraKenigSaloUhlmann2009} as a class of manifolds admitting the existence of \emph{limiting Carleman weights}. In fact, the construction of complex geometrical optics solutions are possible on such manifolds via Carleman estimates approach based on the existence of limiting Carleman weights. Such an approach was introduced by Bukhgeim and Uhlmann \cite{BukhgeimUhlmann2002} and Kenig, Sj\"ostrand and Uhlmann \cite{KenigSjostrandUhlmann2007} in the setting of partial data Calder\'on's inverse conductivity problem in $\R^n$.\smallskip

Our second main result is as follows.

\begin{Theorem}\label{thm::main result 2}
Let $(M,g)$ be a $3$-dimensional admissible manifold and let $4\leq p< 6$. Suppose that $\varepsilon_j,\mu_j\in C^3(M)$ with positive real parts and $a_j,b_j\in C^1(M)$, $j=1,2$. Fix $\omega>0$ outside a discrete set of resonant frequencies $\Sigma\subset\C$ and fix sufficiently small $\epsilon>0$. If
$$
\Lambda_{\varepsilon_1,\mu_1,a_1,b_1}^\omega(f)=\Lambda_{\varepsilon_2,\mu_2,a_2,b_2}^\omega(f)
$$
for all $f\in TW^{1-1/p,p}_{\Div}(\p M)$ with $\|f\|_{TW^{1-1/p,p}_{\Div}(\p M)}<\epsilon$, then $\varepsilon_1=\varepsilon_2$, $\mu_1=\mu_2$, $a_1=a_2$ and $b_1=b_2$ in $M$.
\end{Theorem}

{ Such inverse boundary value problems have been considered for various semilinear and quasilinear elliptic equations and systems (see \cite{HSu,I1,I2,IN,IS,Su1,Su2,SuU}) based on the linearization approach. 

For the type of nonlinearity of Maxwell's equations in a Kerr type medium, after first order linearization, we can recover $\mu$ and $\varepsilon$ by solving corresponding inverse problem for the linear equation (see \cite{KenigSaloUhlmann2011a}). The difficulty lies in reconstructing the susceptibility parameters $a$ and $b$. By calculating the next term of the asymptotic expansion for the admittance map, one obtains the tangential trace $\mathbf t(H_2)$ of the solution to \eqref{eqn::nonlinear Maxwell equation with j-indice}. It carries the energy generated by the nonlinear source $(a|E_1|^2E_1, b|H_1|^2H_1)$, where $(E_1, H_1)$ is the solution to the linear equation. By polarization, we are able to recover $a$ and $b$ from such energy using enough proper solutions. {The solutions we apply here are complex geometrical optics solutions constructed on an admissible manifold as in \cite{KenigSaloUhlmann2011a} or \cite{ChungOlaSaloTzou2016} with proper regularity.  }\\

The paper is organized as following. In Section \ref{sec:prelim}, we present basic facts on differential forms, trace operators, the type of Sobolev spaces and their properties used in this paper. After proving the well-posedness of the direct problem (Theorem \ref{thm::main result 1}) in Section \ref{sec:directproblem}, we compute the asymptotic expansion of the admittance map $\Lambda^\omega_{\varepsilon,\mu,a,b}$ in Section \ref{sec:asymp}. To solve the inverse problem, the reconstruction of $\mu$ and $\varepsilon$ is given in Section \ref{sec:mueps}, and the reconstruction of $a$ and $b$ is given in Section \ref{sec:ab}. The CGO solution is constructed in Section \ref{sec:CGO}. \\
}

\textbf{Acknowledgements.} The research of TZ is supported by the NSF grant DMS-1501049. YA is grateful to Professor Petar Topalov for helpful discussions on direct and inverse problems of nonlinear equations.


\section{Preliminaries}\label{sec:prelim}

In this section we briefly present basic facts on differential forms and trace operators. For more detailed exposition we refer the reader to the manuscript of Schwarz \cite{Schwarz1995}.\smallskip

Let $(M,g)$ be a compact oriented $n$-dimensional Riemannian manifold with smooth boundary. The inner product of tangent vectors with respect to the metric $g$ is denoted by $\<\cdot,\cdot\>_g$, and $|\cdot|_g$ is the notation for the corresponding norm. By $|g|$ we denote the determinant of $g=(g_{ij})$ and $(g^{ij})$ is the inverse matrix of $(g_{ij})$. Finally, there is the induced metric $\imath^*g$ on $\p M$ which gives a rise to the inner product $\<\cdot,\cdot\>_{\imath^*g}$ of vectors tangent to $\p M$.

\subsection{Basic notations for differential forms}\label{sec::2.1}
In what follows, for $F$ some function space ($C^k$, $L^p$, $W^{k,p}$, etc.), we denote by $F\Omega^m(M)$ the corresponding space of $m$-forms. In particular, the space of smooth $m$-forms is denoted by $C^\infty\Omega^{m}(M)$. Let $*:C^\infty\Omega^{m}(M)\to C^\infty\Omega^{n-m}(M)$ be the Hodge star operator. For real valued $\eta,\zeta\in C^\infty\Omega^{m}(M)$, the inner product with respect to $g$ is defined as
\begin{equation}\label{eqn::definition and hodge star isometry property of g-inner product}
\<\eta,\zeta\>_g=*(\eta\wedge*\zeta)=\<*\eta,*\zeta\>_g.
\end{equation}
Its local coordinates expression is $\<\eta,\zeta\>_g=g^{i_1 j_1}\cdots g^{i_m j_m}\eta_{i_1\dots i_m}\zeta_{j_1\dots j_m}$. This can be extended as a bilinear form on complex valued forms on $M$. We also write $|\eta|_g^2=\<\eta,\overline\eta\>_g$.\smallskip

The inner product on $L^2\Omega^{m}(M)$ is defined as
$$
(\eta|\zeta)_{L^2\Omega^{m}(M)}=\int_M\<\eta,\overline{\zeta}\>_g\,d\Vol_g=\int_M\eta\wedge*\overline{\zeta},\qquad \eta,\zeta\in L^2\Omega^m(M),
$$
where $d\Vol_g=* 1=|g|^{1/2}\,dx^1\wedge\cdots\wedge\,dx^n$ is the volume form. The corresponding norm is $\|\cdot\|_{L^2\Omega^{m}(M)}^2=(\cdot|\cdot)_{L^2\Omega^{m}(M)}$. Using the definition of the Hodge star operator $*$, it is not difficult to check that
\begin{equation}\label{eqn::Hodge star is isometry in L^2}
(\eta|\zeta)_{L^2\Omega^{m}(M)}=(*\eta|*\zeta)_{L^2\Omega^{n-m}(M)}.
\end{equation}
Let $d:C^\infty\Omega^{m}(M)\to C^\infty\Omega^{m+1}(M)$ be the external differential. Then the codifferential $\delta:C^\infty\Omega^{m}(M)\to C^\infty\Omega^{m-1}(M)$ is defined as
$$
(d\eta|\zeta)_{L^2\Omega^{m}(M)}=(\eta|\delta\zeta)_{L^2\Omega^{m-1}(M)}
$$
for all  $\eta\in C^\infty_0\Omega^{m-1}(M^{\rm int})$, $\zeta\in C^\infty\Omega^{m}(M)$. The Hodge star operator $*$ and the codifferential $\delta$ have the following properties when acting on $C^\infty\Omega^{m}(M)$:
\begin{equation}\label{eqn::delta in terms of * and d}
*^2=(-1)^{m(n-m)},\quad \delta=(-1)^{m(n-m)-n+m-1}*(d*\cdot).
\end{equation}

For a given $\xi\in C^\infty\Omega^1(M)$, the interior product $i_\xi:C^\infty\Omega^{m}(M)\to C^\infty\Omega^{m-1}(M)$ is the contraction of differential forms by $\xi$. In local coordinates,
$$
i_\xi \eta=g^{ij}\xi_i\,\eta_{ji_1\dots i_{m-1}},\quad \eta\in C^\infty\Omega^{m}(M).
$$
It is the formal adjoint of $\xi$, in the inner product $\<\cdot,\cdot\>_g$ on real valued forms, and has the following expression
\begin{equation}\label{eqn::expression for contraction}
i_\xi\eta=(-1)^{n(m-1)}*(\xi\wedge*\eta),\quad\eta\in C^\infty\Omega^{m}(M).
\end{equation}

The Hodge Laplacian acting on $C^\infty\Omega^{m}(M)$ is defined by $-\Delta=d\delta+\delta d$.\smallskip

Finally, the inner product on $L^2\Omega^m(\p M)$ is given by
$$
(u|v)_{L^2\Omega^{m}(\p M)}=\int_{\p M}\<u,\overline v\>_{\imath^* g}\,d\sigma_{\p M},\qquad u,v\in L^2\Omega^m(\p M),
$$
where $\<\cdot,\cdot\>_{\imath^*g}$ is extended as a bilinear form on complex forms on $\p M$, and $d\sigma_{\p M}=\imath^*(i_\nu d\Vol_g)$ is the volume form on $\p M$ induced by $d\Vol_g$.

\subsection{Integration by parts}
The outward unit normal $\nu$ to $\p M$ can be extended to a vector field near $\p M$ by parallel transport along normal geodesics, initiating from $\p M$ in the direction of $-\nu$, and then to a vector field on $M$ via a cutoff function.\smallskip

The following simple result from~\cite[Lemma~2.1]{Assylbekov2019maxwell} will be used in formulating integration by parts formula in appropriate way.

\begin{Lemma}\label{lemma::expression for boundary term of Stokes' theorem}
If $\eta\in C^\infty\Omega^m(M)$ and $\zeta\in C^\infty\Omega^{m+1}(M)$, then for an open subset $\Gamma\subset \p M$ the following holds
$$
(\mathbf{t}(\eta)|\mathbf{t}(i_\nu  \zeta))_{L^2\Omega^{m}(\Gamma)}=\int_{\Gamma}\mathbf{t}(\eta\wedge*\overline\zeta).
$$
\end{Lemma}

For $\eta\in C^\infty\Omega^{m}(M)$ and $\zeta\in C^\infty\Omega^{m+1}(M)$, using Stokes' theorem, Lemma~\ref{lemma::expression for boundary term of Stokes' theorem} (with~$\Gamma=\p M$) and \eqref{eqn::delta in terms of * and d}, we have the following integration by parts formula for $d$ and $\delta$
\begin{equation}\label{eqn::integration by parts for d and delta}
(\mathbf{t}(\eta)|\mathbf{t}(i_\nu \zeta))_{L^2\Omega^{m}(\p M)}=(d\eta|\zeta)_{L^2\Omega^{m+1}(M)}-(\eta|\delta\zeta)_{L^2\Omega^{m}(M)}.
\end{equation}

\subsection{Extensions of trace operators}\label{section::extensions of t and t i_nu}
The tangential trace operator $\mathbf{t}$ has an extension to a bounded operator from $W^{1,p}\Omega^{m}(M)$ to $W^{1-1/p,p}\Omega^{m}(\p M)$ for $p> 1$. Moreover, for every $f\in W^{1-1/p,p}\Omega^{m}(\p M)$, there is $u\in W^{1,p}\Omega^{m}(M)$ such that $\mathbf{t}(u)=f$ and
$$
\|u\|_{W^{1,p}\Omega^{m}(M)}\le C \|f\|_{W^{1-1/p,p}\Omega^{m}(\p M)};
$$
see \cite[Theorem~1.3.7]{Schwarz1995} and comments.\smallskip

Next, the operator $\mathbf{t}(i_\nu\,\cdot\,)$ is bounded from $W^{1,p}\Omega^{m}(M)$ to $W^{1-1/p,p}\Omega^{m-1}(\p M)$. Moreover, for every $h\in W^{1-1/p,p}\Omega^{m-1}(\p M)$, there is $\zeta\in W^{1,p}\Omega^{m}(M)$ such that $\mathbf{t}(i_\nu\zeta)=h$ and
$$
\|\zeta\|_{W^{1,p}\Omega^{m}(M)}\le C \|h\|_{W^{1-1/p,p}\Omega^{m-1}(\p M)}.
$$
In fact, we can take $\zeta=\nu\wedge w$, where $w\in H^1\Omega^{m-1}(M)$ such that $\mathbf{t}(w)=h$ and $\|w\|_{W^{1,p}\Omega^{m-1}(M)}\le C\|h\|_{W^{1-1/p,p}\Omega^{m-1}(\p M)}$.\smallskip

Finally, if $f\in W^{1-1/p,p}\Omega^m(\p M)$ and $h\in W^{1-1/p,p}\Omega^{m-1}(\p M)$, there is $\xi\in W^{1,p}\Omega^m(M)$ such that $\mathbf t(\xi)=f$, $\mathbf t(i_\nu \xi)=h$ and 
$$
\|\xi\|_{W^{1,p}\Omega^m(M)}\le C \|f\|_{W^{1-1/p,p}\Omega^{m}(\p M)}+C \|h\|_{W^{1-1/p,p}\Omega^{m-1}(\p M)}.
$$
This time, we can take $\xi=(u-\nu\wedge i_\nu u)+\nu\wedge i_\nu\zeta$, where $u\in W^{1,p}\Omega^{m}(M)$ such that $\mathbf{t}(u)=f$ and $\|u\|_{W^{1,p}\Omega^{m}(M)}\le C \|f\|_{W^{1-1/p,p}\Omega^{m}(\p M)}$ and $\zeta\in W^{1,p}\Omega^{m}(M)$ such that $\mathbf{t}(i_\nu\zeta)=h$ and $\|\zeta\|_{W^{1,p}\Omega^{m}(M)}\le C \|h\|_{W^{1-1/p,p}\Omega^{m-1}(\p M)}$.

\subsection{Surface divergence}\label{section::surface divergence} When $n=3$, we define the surface divergence of $f\in W^{1-1/p,p}\Omega^1(\p M)$, for $p>1$, by
$$
\Div(f)=\<d_{\p M}f,d\sigma_{\p M}\>_{\imath^* g}.
$$
If $u\in W^{1,p}\Omega^1(M)$ is arbitrary such that $\mathbf t(u)=f$, then for all $h\in C^\infty(M)$
\begin{align*}
(\Div(f)|h)_{L^2(\p M)}&=\int_{\p M}\<d_{\p M}f,\overline h\,d\sigma_{\p M}\>_{\imath^* g}\,d\sigma_{\p M}=\int_{\p M}\<\mathbf t(du),\mathbf t(\overline h\,i_\nu d\Vol_{g})\>_{g}\,d\sigma_{\p M}\\
&=\int_{\p M}\<\mathbf t(du),\mathbf t(\overline h\,i_\nu *1)\>_{g}\,d\sigma_{\p M}=\int_{\p M}\<\mathbf t(i_\nu *du),\overline h\>_{g}\,d\sigma_{\p M}.
\end{align*}
In the last step we used Lemma~\ref{lemma::expression for boundary term of Stokes' theorem} twice. Thus, we have
\begin{equation}\label{eqn::expression for Div}
\Div(f)=i_\nu*du|_{\p M}
\end{equation}
for all $u\in W^{1,p}\Omega^1(M)$ with $\mathbf t(u)=f$.

\subsection{Technical estimate} We finish this section with the following lemma which ensures that nonlinear terms in the Maxwell equations \eqref{eqn::Maxwell} will be in appropriate functional spaces.

\begin{Lemma}\label{lemma::product estimate}
Let $(M,g)$ be a compact $n$-dimensional Riemannian manifold. If $u\in W^{1,p}\Omega^1(M)$ for $p>n$, then
$$
\||u|_g^{2}u\|_{W^{1,p}\Omega^1(M)}\le C\|u\|_{W^{1,p}\Omega^1(M)}^{3}.
$$
\end{Lemma}
\begin{proof}
To prove the lemma, we first observe that the $W^{1,p}\Omega^m(M)$-norm may be expressed invariantly as
$$
\|f\|_{W^{1,p}\Omega^1(M)}=\|f\|_{L^p\Omega^1(M)}+\|\,|\nabla f|_{g}\|_{L^p(M)},
$$
where $\nabla$ is the Levi-Civita connection defined on tensors on $M$ and $|T|_g$ is the norm of a tensor $T$ on $M$ with respect to the metric $g$.\smallskip

For a given $u\in W^{1,p}\Omega^1(M)$,
\begin{align*}
\||u|_g^{2}u\|_{W^{1,p}\Omega^1(M)}&\le C\||u|_g^{2}u\|_{L^p\Omega^1(M)}+C\|\,|\nabla\big(|u|_g^{2}u\big)|_g\|_{L^p(M)}\\
&\le C\||u|_g^{2}u\|_{L^p\Omega^1(M)}+C\||u|_g^{2}|\nabla u|_g\|_{L^p(M)}\le C\|u\|_{L^\infty\Omega^1(M)}^{2}\|u\|_{W^{1,p}\Omega^1(M)}.
\end{align*}
Since $p>n$ and $M$ is compact, we can use the Sobolev embedding $W^{1,p}\Omega^1(M)\inclusion C\Omega^1(M)$ (\cite[Theorem~1.3.6]{Schwarz1995}), which implies the desired estimate
\end{proof}


\subsection{{ Properties of $W_d^p\Omega^m(M)$ and $W_\delta^p\Omega^m(M)$ spaces, $p>1$}}

\label{section::on H_d and H_delta spaces}

Let $(M,g)$ be a compact oriented $n$-dimensional Riemannian manifold with smooth boundary. In this paper we work with the Banach spaces $W^p_d\Omega^{m}(M)$ and $W^p_\delta\Omega^{m}(M)$, $p>1$, which are the largest domains of $d$ and $\delta$, respectively, acting on $m$-forms:
\begin{align*}
W^p_d\Omega^{m}(M)&:=\{w\in L^p\Omega^{m}(M):dw\in L^p\Omega^{m+1}(M)\},\\
W^p_\delta\Omega^{m}(M)&:=\{u\in L^p\Omega^{m}(M):\delta u\in L^p\Omega^{m-1}( M)\}
\end{align*}
endowed with the norms
\begin{align*}
\|w\|_{W^p_{d}\Omega^{m}(M)}^2&:=\|w\|_{L^p\Omega^m(M)}+\|dw\|_{L^p\Omega^{m+1}(M)},\\
\|u\|_{W^p_{\delta}\Omega^{m}(M)}^2&:=\|u\|_{L^p\Omega^m(M)}+\|\delta u\|_{L^p\Omega^{m-1}(M)}.
\end{align*}
We also use the notations $H_d\Omega^m(M)=W^2_d\Omega^{m}(M)$ and $H_\delta\Omega^m(M)=W^2_\delta\Omega^{m}(M)$, { together with their corresponding traces $TH_d\Omega^{m}(M)$ and $TH_\delta\Omega^{m}(M)$.}\smallskip

In the present section we prove some important properties of these spaces, which were proven in \cite[Section~3]{Assylbekov2019maxwell} for the case $p=2$; see also \cite{Costabel1990,KirschHettlich2015,Monk2003}.\\


First, we show that there are bounded extensions $\mathbf{t}:W^p_d\Omega^{m}(M)\to W^{-1/p,p}\Omega^{m}(\p M)$ and $\mathbf{t}(i_\nu\,\cdot\,):W^p_\delta\Omega^{m+1}(M)\to W^{-1/p,p}\Omega^{m}(\p M)$.\smallskip

Let $(\cdot|\cdot)_{\p M}$ denotes the distributional duality on $\p M$ naturally extending $(\cdot|\cdot)_{L^2\Omega^m(\p M)}$.

\begin{Proposition}\label{prop::tangential trace operator is extended to H_d}
{\rm (a)} The operator $\mathbf{t}:W^{1,p}\Omega^m(M)\to W^{1-1/p,p}\Omega^m(\p M)$ has its extension to a bounded operator $\mathbf{t}:W^p_d\Omega^m(M)\to W^{-1/p,p}\Omega^m(\p M)$ and the following integration by parts formula holds
$$
(\mathbf{t}(\eta)|\mathbf{t}(i_\nu \zeta))_{\p M}=(d\eta|\zeta)_{L^2\Omega^{m+1}(M)}-(\eta|\delta\zeta)_{L^2\Omega^{m}(M)}
$$
for all $\eta\in W^p_d\Omega^m(M)$ and $\zeta\in W^{1,p'}\Omega^{m+1}(M)$, where $p'=p/(p-1)$.\smallskip

{\rm (b)} The operator $\mathbf{t}(i_\nu\,\cdot\,):W^{1,p}\Omega^{m+1}(M)\to W^{1-1/p,p}\Omega^{m}(\p M)$ has its extension to a bounded operator $\mathbf{t}(i_\nu\,\cdot\,):W^p_\delta\Omega^{m+1}(M)\to W^{-1/p,p}\Omega^{m}(\p M)$ and the following integration by parts formula holds
$$
(\mathbf{t}(i_\nu \zeta)|\mathbf{t}(\eta))_{\p M}=(\zeta|d\eta)_{L^2\Omega^{m+1}(M)}-(\delta \zeta|\eta)_{L^2\Omega^{m}(M)}
$$
for all $\zeta\in W^p_\delta\Omega^{m+1}(M)$ and $\eta\in W^{1,p'}\Omega^m(M)$.
\end{Proposition}


\begin{proof}
Let us first prove part (a). Let $w\in C^\infty\Omega^{m}(M)$ and $f\in W^{1/p,p'}\Omega^{m}(\p M)$, where $p'=p/(p-1)$. Then using integration by parts formula~\eqref{eqn::integration by parts for d and delta}, we have
\begin{align*}
(\mathbf{t}(w)|f)_{L^2\Omega^{m}(\p M)}&=(\mathbf{t}(w)|\mathbf{t}(i_\nu \zeta))_{L^2\Omega^{m}(\p M)}\\
&=(dw|\zeta)_{L^2\Omega^{m+1}(M)}-(w|\delta\zeta)_{L^2\Omega^{m}(M)},
\end{align*}
where $\zeta\in W^{1,p'}\Omega^{m+1}(M)$ such that $\mathbf{t}(i_\nu\zeta)=f$ and $\|\zeta\|_{W^{1,p'}\Omega^{m+1}(M)}\le C \|f\|_{W^{1/p,p'}\Omega^{m}(\p M)}$. Then using H\"older's inequality, we show
\begin{align*}
|(\mathbf{t}(w)|f)_{L^2\Omega^{m}(\p M)}|&\le C\|w\|_{W_d^p\Omega^m(M)}\|\zeta\|_{W^{1,p'}\Omega^{m+1}(M)}\\
&\le C\|w\|_{W^p_d\Omega^m(M)}\|f\|_{W^{1/p,p'}\Omega^{m}(\p M)}.
\end{align*}
Therefore, $\mathbf{t}$ can be extended to a bounded operator $W^p_d\Omega^m(M)\to W^{-1/p,p}\Omega^m(\p M)$. In fact, if $\eta\in W^p_d\Omega^m(M)$, then we define $\mathbf{t}(\eta)$ as
$$
(\mathbf{t}(\eta)|\mathbf{t}(i_\nu \zeta))_{\p M}=(d\eta|\zeta)_{L^2\Omega^{m+1}(M)}-(\eta|\delta\zeta)_{L^2\Omega^{m}(M)},
$$
where $\zeta\in W^{1,p'}\Omega^{m+1}(M)$.\smallskip

Now we prove part (b). Let $w\in C^\infty\Omega^{m+1}(M)$ and $f\in W^{1/p,p'}\Omega^{m}(\p M)$. Then using integration parts formula~\eqref{eqn::integration by parts for d and delta}, we have
\begin{align*}
(\mathbf{t}(i_\nu w)|f)_{L^2\Omega^{m}(\p M)}&=(\mathbf{t}(i_\nu w)|\mathbf{t}(u))_{L^2\Omega^{m}(\p M)}\\
&=(w|du)_{L^2\Omega^{m+1}(M)}-(\delta w|u)_{L^2\Omega^{m}(M)},
\end{align*}
where $u\in W^{1,p'}\Omega^{m}(M)$ such that $\mathbf{t}(u)=f$ and $\|u\|_{W^{1,p'}\Omega^{m}(M)}\le C \|f\|_{W^{1/p,p'}\Omega^{m}(\p M)}$. Therefore, using H\"older's inequality, we can estimate
\begin{align*}
|(\mathbf{t}(i_\nu w)|f)_{L^2\Omega^{m}(\p M)}|&\le C\|w\|_{W^p_\delta \Omega^{m+1}(M)}\|u\|_{W^{1,p'}\Omega^{m}(M)}\\
&\le C\|w\|_{W_\delta^p\Omega^{m+1}(M)}\|f\|_{W^{1/p,p'}\Omega^{m}(\p M)}.
\end{align*}
Thus, $\mathbf{t}(i_\nu\,\cdot\,)$ can be extended to a bounded operator $W^p_\delta\Omega^{m+1}(M)\to W^{-1/p,p}\Omega^m(\p M)$. In fact, if $\zeta\in W^p_\delta\Omega^{m+1}(M)$ we define $\mathbf{t}(i_\nu \zeta)$ as
$$
(\mathbf{t}(i_\nu \zeta)|\mathbf{t}(\eta))_{\p M}=(\zeta|d\eta)_{L^2\Omega^{m+1}(M)}-(\delta \zeta|\eta)_{L^2\Omega^{m}(M)},
$$
where $\eta\in W^{1,p'}\Omega^m(M)$.
\end{proof}

We will also need the following embedding results. For $p=2$, these were proven in Euclidean and Riemannian settings \cite{Assylbekov2019maxwell,KirschHettlich2015,Monk2003}.

\begin{Proposition}\label{prop::bounded embedding of W^p_{d,delta} with regular tangential trace into W^1p}
Suppose that $p>1$, $u\in W^p_{d}\Omega^m(M)\cap W^p_\delta\Omega^m(M)$ and $\mathbf t(u)\in W^{1-1/p}\Omega^m(\p M)$. Then $u\in W^{1,p}\Omega^m(M)$ and
$$
\|u\|_{W^{1,p}\Omega^m(M)}\le C\big(\|u\|_{W^p_d\Omega^m(M)}+\|\delta u\|_{L^p\Omega^{m-1}(M)}+\|\mathbf t(u)\|_{W^{1-1/p,p}\Omega^m(\p M)}\big)
$$
for some constant $C>0$ independent of $u$.
\end{Proposition}

In Euclidean setting, this was proven in the case $m=1$ and $p=2$ by Costabel \cite{Costabel1990}; see also \cite{KirschHettlich2015,Monk2003}. On manifolds, for the case $p=2$ and for arbitrary $m$, this was proved in \cite{Assylbekov2019maxwell}.\smallskip

Write
$$
\mathcal H^{m}_D(M):=\{u\in W^{1,2}\Omega^m(M):du=0,\quad \delta u=0,\quad \mathbf t(u)=0\}.
$$
Proposition~\ref{prop::bounded embedding of W^p_{d,delta} with regular tangential trace into W^1p} is based on the following result from \cite{Schwarz1995}.

\begin{Lemma}\label{lemma::from Schwarz thm3.2.5}
Let $k\ge 0$ be an integer and let $p>1$. Given $w\in W^{k,p}\Omega^{m+1}(M)$, $v\in W^{k,p}\Omega^{m-1}(M)$ and $h\in W^{k+1,p}\Omega^m(M)$, there is a unique $\psi\in W^{k+1,p}\Omega^m(M)$, up to a form in $\mathcal H_D^m(M)$, that solves
$$
d\psi=w,\quad \delta \psi=v,\quad \mathbf{t}(\psi)=\mathbf{t}(h)
$$
if and only if
$$
dw=0,\quad \mathbf t(w)=\mathbf t(dh),\quad \delta v=0
$$
and
$$
(w|\chi)_{L^2\Omega^{m+1}(M)}=(\mathbf t(h)|\mathbf t(i_\nu \chi))_{L^2\Omega^m(\p M)},\quad (v|\lambda)_{L^2\Omega^{m-1}(M)}=0
$$
for all $\chi\in\mathcal H^{m+1}_D(M)$, $\lambda\in\mathcal H^{m-1}_D(M)$. Moreover, $\psi$ satisfies the estimate
\begin{align*}
\|\psi\|_{W^{k+1,p}\Omega^m(M)}\le &C\big(\|w\|_{W^{k,p}\Omega^{m+1}(M)}+\|v\|_{W^{k,p}\Omega^{m-1}(M)}\big)\\
&+C\big(\|\mathbf t(h)\|_{W^{k+1-1/p,p}\Omega^m(\p M)}+\|\mathbf t(*h)\|_{W^{k+1-1/p,p}\Omega^{n-m}(\p M)}\big).
\end{align*}
\end{Lemma}
\begin{proof}
Follows from \cite[Theorem~3.2.5]{Schwarz1995}.
\end{proof}

The proof of Proposition~\ref{prop::bounded embedding of W^p_{d,delta} with regular tangential trace into W^1p} is identical to the proof of \cite[Proposition~3.2]{Assylbekov2019maxwell} (case $p=2$), but for different integrability spaces. Therefore, we do not include it here. We only mention that the use of Lemma~\ref{lemma::from Schwarz thm3.2.5} is crucial and similar ideas were used in the next proposition, after certain modifications.

\begin{Proposition}\label{prop::bounded embedding of W^p_{d,delta} with regular normal trace into W^1p}
Suppose that $p>1$, $u\in W^p_{d}\Omega^m(M)\cap W^p_\delta\Omega^m(M)$ and $\mathbf t(i_\nu u)\in W^{1-1/p,p}\Omega^{m-1}(\p M)$. Then $u\in W^{1,p}\Omega^m(M)$ and
$$
\|u\|_{W^{1,p}\Omega^m(M)}\le C\big(\|u\|_{W^p_d\Omega^m(M)}+\|\delta u\|_{L^p\Omega^{m-1}(M)}+\|\mathbf t(i_\nu u)\|_{W^{1-1/p,p}\Omega^{m-1}(\p M)}\big)
$$
for some constant $C>0$ independent of $u$.
\end{Proposition}

\begin{proof}
Since $\mathbf{t}(i_\nu u)\in W^{1-1/p,p}\Omega^{m-1}(\p M)$, by discussion in Section~\ref{section::extensions of t and t i_nu} there is $\eta\in W^{1,p}\Omega^m(M)$ such that $\mathbf{t}(\eta)=0$, $\mathbf t(i_\nu \eta)=\mathbf t(i_\nu u)$ and
$$
\|\eta\|_{W^{1,p}\Omega^m(M)}\le C\|\mathbf t(i_\nu \eta)\|_{W^{1-1/p,p}\Omega^{m-1}(\p M)}=C\|\mathbf t(i_\nu u)\|_{W^{1-1/p,p}\Omega^{m-1}(\p M)}.
$$
Set $h=*\eta$, then $h\in W^{1,p}\Omega^{n-m}(M)$. Using boundedness of $\mathbf t:W^{1,p}\Omega^{n-m}(M)\to W^{1-1/p,p}\Omega^{n-m}(\p M)$ and the above estimate,
\begin{equation}\label{eqn::estimate for the trace of h in the proof of embedding result}
\begin{aligned}
\|\mathbf t(h)\|_{W^{1-1/p,p}\Omega^{n-m}(\p M)}&\le C\|h\|_{W^{1,p}\Omega^{n-m}(M)}\\
&\le C\|\eta\|_{W^{1,p}\Omega^{m}(M)}\\
&\le C\|\mathbf t(i_\nu u)\|_{W^{1-1/p,p}\Omega^{m-1}(\p M)}.
\end{aligned}
\end{equation}
Set $\tilde u:=*u$, then it is clear that $\tilde u\in W^{p}_d\Omega^{n-m}(M)\cap W^{p}_\delta\Omega^{n-m}(M)$. Write $w=d\tilde u\in L^p\Omega^{n-m+1}(M)$ and $v=\delta \tilde u\in L^p\Omega^{n-m-1}(M)$.\smallskip

An important fact is that $\mathbf t(\tilde u)=\mathbf t(h)$. Indeed, for arbitrary $\varphi\in W^{1/p,p/(p-1)}\Omega^{n-m}(\p M)$, as discussed in Section~\ref{section::extensions of t and t i_nu}, there is $\zeta\in W^{1,p/(p-1)}\Omega^{n-m+1}(M)$ such that $\mathbf t(i_\nu \zeta)=\varphi$. Then, using integration by parts formulas in Proposition~\ref{prop::tangential trace operator is extended to H_d}, we get
\begin{align*}
(\mathbf t(\tilde u-h)|\varphi)_{\p M}&=(\mathbf t(*(u-\eta))|\mathbf t(i_\nu \zeta))_{\p M}\\
&=(d*(u-\eta)|\zeta)_{L^2\Omega^{n-m+1}(M)}-(*(u-\eta)|\delta\zeta)_{L^2\Omega^{n-m}(M)}\\
&=(\delta(u-\eta)|*\zeta)_{L^2\Omega^{m-1}(M)}-(u-\eta|d*\zeta)_{L^2\Omega^{m-1}(M)}\\
&=-(\mathbf t(i_\nu (u-\eta))|\mathbf t(*\zeta))_{\p M}=0,
\end{align*}
since $\mathbf t(i_\nu \eta)=\mathbf t(i_\nu u)$. Therefore, $\mathbf t(\tilde u)=\mathbf t(h)$.\smallskip

We wish to use Lemma~\ref{lemma::from Schwarz thm3.2.5}, and hence we need to show that $w$, $v$ and $h$ satisfy the hypothesis of Lemma~\ref{lemma::from Schwarz thm3.2.5}. Obviously, we have $dw=0$ and $\delta v=0$. Integrating by parts and using that $\mathbf t(\tilde u)=\mathbf t(h)$, we can show that for all $\chi\in\mathcal H^{n-m+1}_D(M)$
$$
(w|\chi)_{L^2\Omega^{n-m+1}(M)}=(d\tilde u|\chi)_{L^2\Omega^{n-m+1}(M)}=(\mathbf{t}(h)|\mathbf{t}(i_\nu\chi))_{L^2\Omega^{n-m}(\p M)}.
$$
Similary for all $\lambda\in\mathcal H^{n-m-1}_D(M)$, using the integration by parts formula in part (b) of Proposition~\ref{prop::tangential trace operator is extended to H_d}, we can show that
$$
(v|\lambda)_{L^2\Omega^{n-m-1}(M)}=(\delta \tilde u|\lambda)_{L^2\Omega^{n-m-1}(M)}=-(\mathbf t(i_\nu \tilde u)|\mathbf t(\lambda))_{\p M}=0.
$$
Next, we show that $\mathbf t(w)=\mathbf t(dh)$. For arbitrary $\varphi\in W^{1/p,p/(p-1)}\Omega^{n-m+1}(\p M)$, as discussed in Section~\ref{section::extensions of t and t i_nu}, there is $\zeta\in W^{1,p/(p-1)}\Omega^{n-m+2}(M)$ such that $\mathbf t(i_\nu \zeta)=\varphi$. Then, using integration by parts formulas in Proposition~\ref{prop::tangential trace operator is extended to H_d}, we get
$$
(\mathbf t(w)|\varphi)_{\p M}=(\mathbf t(d\tilde u)|\mathbf t(i_\nu \zeta))_{\p M}=-(d\tilde u|\delta\zeta)_{L^2\Omega^{n-m+1}(M)}=-(\mathbf t(\tilde u)|\mathbf t(i_\nu \delta \zeta))_{\p M}.
$$
Since $\mathbf t(\tilde u)=\mathbf t(h)$, using integration by parts formulas in Proposition~\ref{prop::tangential trace operator is extended to H_d}, gives
$$
(\mathbf t(w)|\varphi)_{\p M}=-(\mathbf t(h)|\mathbf t(i_\nu \delta \zeta))_{\p M}=-(dh|\delta\zeta)_{L^2\Omega^{n-m+1}(M)}=(\mathbf t(dh)|\varphi)_{\p M},
$$
which implies $\mathbf t(w)=\mathbf t(dh)$.\smallskip

Thus, all hypotheses of Lemma~\ref{lemma::from Schwarz thm3.2.5} are satisfied for $w$, $v$ and $h$. Hence, we find $\psi\in W^{1,p}\Omega^{n-m}(M)$ such that $d\psi=w$, $\delta\psi=v$ and ${\mathbf{t}}(\psi)={\mathbf{t}}(h)={\mathbf{t}}(\tilde u)$ and satisfying
\begin{align*}
\|\psi\|_{W^{1,p}\Omega^{n-m}(M)}\le &C\big(\|w\|_{L^p\Omega^{n-m+1}(M)}+\|v\|_{L^p\Omega^{n-m-1}(M)}\big)\\
&+C\big(\|\mathbf t(h)\|_{W^{1-1/p,p}\Omega^{n-m}(\p M)}+\|\mathbf t(*h)\|_{W^{1-1/p,p}\Omega^{m}(\p M)}\big).
\end{align*}
Using \eqref{eqn::estimate for the trace of h in the proof of embedding result}, $\mathbf t(*h)=\mathbf t(\eta)=0$, $w=d*u$ and $v=\delta *u$, we get
$$
\|\psi\|_{W^{1,p}\Omega^{n-m}(M)}\le C\big(\|u\|_{W^p_d\Omega^{m}(M)}+\|\delta u\|_{L^p\Omega^{m-1}(M)}+\|\mathbf t(i_\nu u)\|_{W^{1-1/p,p}\Omega^{m-1}(\p M)}\big).
$$
Write $\rho=\tilde u-\psi$, then $d\rho=0$ and $\delta \rho=0$. Therefore, $\rho$ solves $-\Delta \rho=0$ with $\mathbf t(\rho)=0$, $\mathbf t(\delta \rho)=0$. By \cite[Theorem~2.2.4]{Schwarz1995}, it follows that $\rho=0$. Since $\tilde u=*u$, the last estimate together with \eqref{eqn::estimate for the trace of h in the proof of embedding result} clearly implies the result.
\end{proof}


\section{{ Well-posedness of the direct problem}}\label{sec:directproblem}

\subsection{  Direct problem for linear equations}
To prove existence and uniqueness result for nonlinear equations, we first need to study the direct problem for linear equations.

\begin{Theorem}\label{thm::well posedness new version homogeneous}
Let $2\leq p\leq 6$ and let $\varepsilon,\mu\in C^1(M)$ be complex functions with positive real parts. There is a discrete subset $\Sigma$ of $\C$ such that for all $\omega\notin \Sigma$ and for a given $f\in TW^{1-1/p,p}_{\Div}(\p M)$ the Maxwell's equation 
\begin{equation}\label{eqn::Maxwell homogeneous in appendix}
* dE=i\omega\mu H,\quad
* dH=-i\omega \varepsilon E
\end{equation}
has a unique solution $(E,H)\in W^{1,p}_{\Div}(M)\times W^{1,p}_{\Div}(M)$ satisfying $\mathbf{t}(E)=f$ and 
$$
\|E\|_{W^{1,p}_{\Div}(M)}+\|H\|_{W^{1,p}_{\Div}(M)}\le C\|f\|_{TW^{1-1/p,p}_{\Div}(\p M)},
$$
for some constant $C>0$ independent of $f$.
\end{Theorem}
\begin{proof}
Since $p>1$, the inclusion $TW^{1-1/p,p}_{\Div}(\p M)\inclusion TH_d\Omega^1(\p M)$ is bounded. Then by \cite[Theorem~1.1]{Assylbekov2019maxwell} there is a unique solution $(E,H)\in H_{d}\Omega^1(M)\times H_{d}\Omega^1(M)$ of~\eqref{eqn::Maxwell homogeneous in appendix} such that $\mathbf{t}(E)=f$. 
According to Theorem~\ref{thm::main regularity result}, we have $(E,H)\in H^1_{\Div}(M)\times H^1_{\Div}(M)$. 
By Sobolev embedding, inclusion $W^{1,2}\Omega^1(M)\inclusion L^p\Omega^1(M)$ is bounded for $2\le p\le 6$; see \cite[Theorem~1.3.6 (a)]{Schwarz1995}. Using this together with \eqref{eqn::Maxwell homogeneous in appendix}, we get $(E,H)\in W^p_d\Omega^1(M)\times W^p_d\Omega^1(M)$. 
Recall that $\mathbf{t}(E)=f\in TW^{1-1/p,p}_{\Div}(\p M)$. Then an application of Theorem~\ref{thm::main regularity result} implies that $(E,H)\in W^{1,p}_{\Div}(M)\times W^{1,p}_{\Div}(M)$ and satisfies the estimate stated in the theorem. The proof is complete.
\end{proof}

We also consider the linear non-homogeneous problem. The following well-posedness result will be used in dealing with nonlinear terms of \eqref{eqn::Maxwell}. We define
$$
W^{1,p}_D\Omega^1(M):=\{u\in W^{1,p}\Omega^1(M):\mathbf t(u)=0\}.
$$

\begin{Theorem}\label{thm::well posedness new version}
Let $2\le p\le 6$ and let $\varepsilon,\mu\in C^1(M)$ be complex functions with positive real parts. Suppose that $J_e,J_m\in W^p_\delta\Omega^1(M)$ and $i_\nu J_e|_{\p M}, i_\nu J_m|_{\p M}\in W^{1-1/p,p}(\p M)$. There is a discrete subset $\Sigma$ of $\C$ such that for all $\omega\notin \Sigma$ the Maxwell's system
\begin{equation}\label{eqn::Maxwell in appendix}
* dE=i\omega\mu H+J_m,\quad
* dH=-i\omega \varepsilon E-J_e
\end{equation}
has a unique solution $(E,H)\in W^{1,p}_D\Omega^1(M)\times W^{1,p}_{\Div}(M)$ satisfying
\begin{align*}
\|E\|_{W^{1,p}_{\Div}(M)}+\|H\|_{W^{1,p}_{\Div}(M)}&\le C(\|i_\nu J_e|_{\p M}\|_{W^{1-1/p,p}(\p M)}+\|i_\nu J_m|_{\p M}\|_{W^{1-1/p,p}(\p M)})\\
&\qquad+C(\|J_e\|_{W^{p}_\delta\Omega^1(M)}+\|J_m\|_{W^{p}_\delta\Omega^1(M)})
\end{align*}
for some constant $C>0$ independent of $J_e$ and $J_m$.
\end{Theorem}

\begin{proof}
We follow the similar approach as in the proof of Theorem~\ref{thm::well posedness new version homogeneous}. Since $p\ge 2$, the inclusion $W^{p}_\delta\Omega^1(M)\inclusion L^2\Omega^1(M)$ is bounded. Then by \cite[Theorem~1.2]{Assylbekov2019maxwell} there is a unique solution $(E,H)\in H_{d}\Omega^1(M)\times H_{d}\Omega^1(M)$ of~\eqref{eqn::Maxwell in appendix} such that $\mathbf{t}(E)=0$. 
According to Theorem~\ref{thm::main regularity result}, we have $(E,H)\in H^1_{D}\Omega^1(M)\times H^1_{\Div}(M)$. 
Using Sobolev embedding $H^1\Omega^1(M)\inclusion L^p\Omega^1(M)$ for $2\le p\le 6$ together with \eqref{eqn::Maxwell in appendix}, we get $(E,H)\in W^p_d\Omega^1(M)\times W^p_d\Omega^1(M)$. 
Since $\mathbf{t}(E)=0$, Theorem~\ref{thm::main regularity result} implies that $(E,H)\in W^{1,p}_{D}\Omega^1(M)\times W^{1,p}_{\Div}(M)$ and satisfies the estimate stated in the theorem. The proof is complete.
\end{proof}


\subsection{ Proof of Theorem~\ref{thm::main result 1}}

Suppose $f\in TW^{1-1/p,p}_{\Div}(\p M)$ such that $\|f\|_{TW^{1-1/p,p}_{\Div}(\p M)}<\epsilon$, where $\epsilon>0$ to be determined. By Theorem~\ref{thm::well posedness new version homogeneous}, when $2\leq p\leq 6$, there is a unique $(E_0,H_0)\in W^{1,p}_{\Div}(M)\times W^{1,p}_{\Div}(M)$ solving
$$
* dE_0=i\omega\mu H_0,\quad * dH_0=-i\omega \varepsilon E_0,\quad \mathbf t(E_0)=f
$$
and satisfying
$$
\|E_0\|_{W^{1,p}_{\Div}(M)}+\|H_0\|_{W^{1,p}_{\Div}(M)}\le C\|f\|_{TW^{1-1/p,p}_{\Div}(\p M)}.
$$
Then $(E,H)$ is a solution of \eqref{eqn::Maxwell} if and only if $(E',H')$ defined by $(E,H)=(E_0,H_0)+(E',H')$ satisfies
\begin{equation}\label{eqn::modified nonlinear Maxwell equation with zero boundary data}
\begin{cases}
* dE'=i\omega\mu H'+i\omega b|H_0+H'|_g^{2}(H_0+H'),\\
* dH'=-i\omega \varepsilon E'-i\omega a|E_0+E'|_g^{2}(E_0+E'),\\
\mathbf t(E')=0.
\end{cases}
\end{equation}
By Theorem~\ref{thm::well posedness new version}, there is a bounded and linear operator
$$
\mathcal G_\omega^{\varepsilon,\mu}:W^{1,p}\Omega^1(M)\times W^{1,p}\Omega^1(M)\to W^{1,p}_D\Omega^1(M)\times W^{1,p}_{\Div}(M)
$$
mapping $(J_e,J_m)\in W^{1,p}\Omega^1(M)\times W^{1,p}\Omega^1(M)$ to the unique solution $(\widetilde E,\widetilde H)$ of the problem
$$
* d\widetilde E=i\omega\mu \widetilde H+J_m,\quad * d\widetilde H=-i\omega \varepsilon \widetilde E-J_e,\quad \mathbf t(\widetilde E)=0.
$$
Define $X_\delta$ to be the set of $(e,h)\in W^{1,p}_D\Omega^1(M)\times W^{1,p}_{\Div}(M)$ such that
$$
\|(e,h)\|_{W^{1,p}\Omega^1(M)\times W^{1,p}_{\Div}(M)}:=\|e\|_{W^{1,p}\Omega^1(M)}+\|h\|_{W^{1,p}_{\Div}(M)}\le\delta,
$$
where $\delta>0$ will be determined later. Define an operator $A$ on $X_\delta$ as
$$
A(e,h):=\mathcal G_\omega^{\varepsilon,\mu}\big(i\omega a |E_0+e|_g^{2}(E_0+e)\, ,\, i\omega b|H_0+h|_g^{2}(H_0+h)\big).
$$

We wish to show that for sufficiently small $\epsilon>0$ and $\delta>0$, depending on the frequency $\omega$, the operator $A$ is a contraction on $X_\delta$.\smallskip

First, we show that $A$ maps $X_\delta$ into itself. Using Lemma~\ref{lemma::product estimate}, we can show that when $p>n=3$, for all $(e,h)\in X_\delta$,
\begin{align*}
\|A(e,h)&\|_{W^{1,p}\Omega^1(M)\times W^{1,p}_{\Div}(M)}\\
&\le C\omega\big(\|\,|E_0+e|_g^{2}(E_0+e)\|_{W^{1,p}\Omega^1(M)}+\|\,|H_0+h|_g^{2}(H_0+h)\|_{W^{1,p}\Omega^1(M)}\big)\\
&\le C\omega\big(\|E_0+e\|_{W^{1,p}\Omega^1(M)}^{3}+\|H_0+h\|_{W^{1,p}\Omega^1(M)}^{3}\big)\\
&\le C\omega\big(\|E_0\|_{W^{1,p}\Omega^1(M)}^{3}+\|e\|_{W^{1,p}\Omega^1(M)}^{3}+\|H_0\|_{W^{1,p}\Omega^1(M)}^{3}+\|h\|_{W^{1,p}\Omega^1(M)}^{3}\big).
\end{align*}
Therefore,
\begin{equation}\label{eqn::control for A}
\begin{aligned}
\|A(e,&h)\|_{W^{1,p}\Omega^1(M)\times W^{1,p}_{\Div}(M)}\\
&\le C\omega\epsilon^{2}\|f\|_{TW^{1-1/p,p}_{\Div}(\p M)}+C\omega\delta^{2}\|(e,h)\|_{W^{1,p}\Omega^1(M)\times W^{1,p}_{\Div}(M)}.
\end{aligned}
\end{equation}
In particular, this gives
$$
\|A(e,h)\|_{W^{1,p}\Omega^1(M)\times W^{1,p}_{\Div}(M)}\le C\omega(\epsilon^{3}+\delta^{3}).
$$
Taking $\epsilon>0$ and $\delta>0$ sufficently small, below we will ensure that $A$ maps $X_\delta$ into itself.\smallskip

Next, we show that $A$ is contraction on $X_\delta$. For this we need the following technical~lemma.

\begin{Lemma}\label{lemma::product difference estimate}
Let $(M,g)$ be a compact $n$-dimensional Riemannian manifold and let $p>n$. If $u,v\in W^{1,p}\Omega^1(M)$, then
$$
\|(|u|_g^{2}u-|v|_g^{2}v)\|_{W^{1,p}\Omega^1(M)}\le C(\|u\|_{W^{1,p}\Omega^1(M)}^{2}+\|v\|_{W^{1,p}\Omega^1(M)}^{2})\|u-v\|_{W^{1,p}\Omega^1(M)}.
$$
\end{Lemma}

Assuming this result, we continue the proof of Theorem~\ref{thm::main result 1}. Using Lemma~\ref{lemma::product difference estimate}, we also can show that for all $(e_1,h_1),(e_2,h_2)\in X_\delta$
\begin{align*}
\|A&(e_1,h_1)-A(e_2,h_2)\|_{W^{1,p}\Omega^1(M)\times W^{1,p}_{\Div}(M)}\\
&\le C\omega\|\,|E_0+e_1|_g^{2}(E_0+e_1)-|E_0+e_2|_g^{2}(E_0+e_2)\|_{W^{1,p}\Omega^1(M)}\\
&\quad+C\omega\|\,|H_0+h_1|_g^{2}(H_0+h_1)-|H_0+h_2|_g^{2}(H_0+h_2)\|_{W^{1,p}\Omega^1(M)}\\
&\le C\omega(\|E_0+e_1\|_{W^{1,p}\Omega^1(M)}^{2}+\|E_0+e_2\|_{W^{1,p}\Omega^1(M)}^{2})\|e_1-e_2\|_{W^{1,p}\Omega^1(M)}\\
&\quad+C\omega(\|H_0+h_1\|_{W^{1,p}\Omega^1(M)}^{2}+\|H_0+h_2\|_{W^{1,p}\Omega^1(M)}^{2})\|h_1-h_2\|_{W^{1,p}\Omega^1(M)}\\
&\le C\omega(\|E_0\|_{W^{1,p}\Omega^1(M)}^{2}+\|e_1\|_{W^{1,p}\Omega^1(M)}^{2}+\|e_2\|_{W^{1,p}\Omega^1(M)}^{2})\|e_1-e_2\|_{W^{1,p}\Omega^1(M)}\\
&\quad+C\omega(\|H_0\|_{W^{1,p}\Omega^1(M)}^{2}+\|h_1\|_{W^{1,p}\Omega^1(M)}^{2}+\|h_2\|_{W^{1,p}\Omega^1(M)}^{2})\|h_1-h_2\|_{W^{1,p}\Omega^1(M)}\\
&\le C\omega(\epsilon^{2}+\delta^{2})\big(\|e_1-e_2\|_{W^{1,p}\Omega^1(M)}+\|h_1-h_2\|_{W^{1,p}\Omega^1(M)}\big).
\end{align*}
These imply that $A$ is contraction on~$X_\delta$, if $C\omega(\epsilon^{3}+\delta^{3})\le \delta$ and $C\omega(\epsilon^{2}+\delta^{2})<1$. Now, using the contraction mapping theorem, we find a unique $(E',H')\in X_\delta$ 
such that $A(E',H')=(E',H')$ and hence solving \eqref{eqn::modified nonlinear Maxwell equation with zero boundary data}. Using $A(E',H')=(E',H')$ in \eqref{eqn::control for A} and taking $\delta>0$ sufficently small, one can see that $(E',H')$ satisfies the estimate
$$
\|E'\|_{W^{1,p}\Omega^1(M)}+\|H'\|_{W^{1,p}_{\Div}(M)}\le C\|f\|_{TW^{1-1/p,p}_{\Div}(\p M)}.
$$
Finally, $(E,H)=(E_0,H_0)+(E',H')$ solves \eqref{eqn::Maxwell} with $\mathbf t(E)=f$ and satisfies the estimate
$$
\|E\|_{W^{1,p}_{\Div}(M)}+\|H\|_{W^{1,p}_{\Div}(M)}\le C\|f\|_{TW^{1-1/p,p}_{\Div}(\p M)}.
$$
The proof of Theorem~\ref{thm::main result 1} is thus complete.\smallskip

We emphasize that the requirement $3< p\leq 6$ is due to the Sobolev embedding theorem (see \cite[Theorem~1.3.6 (a)]{Schwarz1995}) used in Lemma \ref{lemma::product estimate}, Lemma \ref{lemma::product difference estimate} and Theorem \ref{thm::well posedness new version homogeneous}.

\begin{proof}[Proof of Lemma~\ref{lemma::product difference estimate}]Recall  that the $W^{1,p}\Omega^m(M)$-norm may be expressed invariantly as
$$
\|f\|_{W^{1,p}\Omega^1(M)}=\|f\|_{L^p\Omega^1(M)}+\|\,|\nabla f|_{g}\|_{L^p(M)},
$$
where $\nabla$ is the Levi-Civita connection defined on tensors on $M$ and $|T|_g$ is the norm of a tensor $T$ on $M$ with respect to the metric $g$.\smallskip

By density of $C^\infty\Omega^1(M)$ in $W^{1,p}\Omega^1(M)$, it is enough to assume that $u,v\in C^\infty\Omega^1(M)$. Recall that
$$
\|(|u|_g^{2}u-|v|_g^{2}v)\|_{W^{1,p}\Omega^1(M)}=\||u|_g^{2}u-|v|_g^{2}v\|_{L^p\Omega^1(M)}+\|\,|\nabla\big(|u|_g^{2}u-|v|_g^{2}v\big)|_{g}\|_{L^p(M)}.
$$
We can write
$$
\nabla\big(|u|_g^{2}u-|v|_g^{2}v\big)=|u|_g^{2}\nabla u-|v|_g^{2}\nabla v+2\Re\<u,\nabla\overline{u}\>_g u-2\Re\<v,\nabla\overline{v}\>_g v.
$$
Therefore,
\begin{equation}\label{eqn::w1p norm of fuu-fvv}
\begin{aligned}
\|(|u|_g^{2}u-|v|_g^{2}v)\|_{W^{1,p}\Omega^1(M)}&\le C\||u|_g^{2}u-|v|_g^{2}v\|_{L^p\Omega^1(M)}+C\|\,|(|u|_g^{2}\nabla u-|v|_g^{2}\nabla v)|_g\|_{L^p(M)}\\
&\qquad+C\|\,|(\Re\<u,\nabla\overline{u}\>_g u-\Re\<v,\nabla\overline{v}\>_g v)|_g\|_{L^p(M)}.
\end{aligned}
\end{equation}
Write $w_\theta=u+\theta(v-u)$. Let us estimate the first term on the right hand-side of~\eqref{eqn::w1p norm of fuu-fvv}. Then
$$
|v|_g^{2}v-|u|_g^{2}u=\int_0^1\frac{\p}{\p \theta}\big\{|w_\theta|_g^{2}w_\theta\big\}\,d\theta=\int_0^1 \{2\Re\<w_\theta,v-u\>_g\,w_\theta+|w_\theta|_g^{2}(v-u)\}\,d\theta
$$
and hence
$$
|(|u|_g^{2}u-|v|_g^{2}v)|_g\le C(\|u\|_{L^\infty\Omega^1(M)}^{2}+\|v\|_{L^\infty\Omega^1(M)}^{2})|u-v|_g.
$$
Therefore, using the Sobolev embedding $W^{1,p}\Omega^1(M)\inclusion C\Omega^1(M)$ as in Lemma~\ref{lemma::product estimate}, we get for $p>n$
\begin{align*}
\||u|_g^{2}u-|v|_g^{2}v\|_{L^p\Omega^1(M)}&\le C(\|u\|_{L^\infty\Omega^1(M)}^{2}+\|v\|_{L^\infty\Omega^1(M)}^{2}) \|u-v\|_{L^p\Omega^1(M)}\\
&\le C(\|u\|_{W^{1,p}\Omega^1(M)}^{2}+\|v\|_{W^{1,p}\Omega^1(M)}^{2}) \|u-v\|_{L^p\Omega^1(M)}.
\end{align*}
Now, we estimate the second term on the right hand-side of \eqref{eqn::w1p norm of fuu-fvv}. Similarly as before, we can show
\begin{align*}
|v|_g^{2}\nabla v-|u|_g^{2}\nabla u&=\int_0^1\frac{\p}{\p \theta}\big\{|w_\theta|_g^{2}\,\nabla w_\theta\big\}\,d\theta=\int_0^1 \{2\Re\<w_\theta,\overline v-\overline u\>_g\,\nabla w_\theta+|w_\theta|_g^{2}\nabla(v-u)\}\,d\theta.
\end{align*}
Then
\begin{align*}
|(|u|_g^{2}\nabla u-|v|_g^{2}\nabla v)|_g&\le C(\|u\|_{L^\infty\Omega^1(M)}+\|v\|_{L^\infty\Omega^1(M)})\|u-v\|_{L^\infty\Omega^1(M)}(|\nabla u|_g+|\nabla v|_g)\\
&\quad+C(\|u\|_{L^\infty\Omega^1(M)}^{2}+\|v\|_{L^\infty\Omega^1(M)}^{2})|\nabla(u-v)|_g.
\end{align*}
Therefore, using the Sobolev embedding $W^{1,p}\Omega^1(M)\inclusion C\Omega^1(M)$, we get
\begin{align*}
\|\,|(|u|_g^{2}\nabla u&-|v|_g^{2}\nabla v)|_g\|_{L^p(M)}\le C(\|u\|_{W^{1,p}\Omega^1(M)}^{2}+\|v\|_{W^{1,p}\Omega^1(M)}^{2})\|u-v\|_{W^{1,p}\Omega^1(M)}.
\end{align*}
Finally, we estimate the last term on the right hand-side of \eqref{eqn::w1p norm of fuu-fvv}. For this, we write
\begin{align*}
\Re\<v,\nabla\overline{v}\>_g v&-\Re\<u,\nabla\overline{u}\>_gu=\int_0^1\frac{\p}{\p \theta}\big\{\Re\<w_\theta,\nabla\overline w_\theta\>_g\,w_\theta\big\}\,d\theta\\
&=\int_0^1 \{\big(\Re\<v-u,\nabla\overline w_\theta\>_g+\Re\<w_\theta,\nabla(\overline v-\overline u)\>_g\big)\,w_\theta+\Re\<w_\theta,\nabla\overline w_\theta\>_g(v-u)\}\,d\theta.
\end{align*}
Therefore,
\begin{multline*}
|(\Re\<u,\nabla\overline{u}\>_g u-\Re\<v,\nabla\overline{v}\>_g v)|_g\\
\le C(\|u\|_{L^\infty\Omega^1(M)}+\|v\|_{L^\infty\Omega^1(M)})\|u-v\|_{L^\infty\Omega^1(M)}(|\nabla u|_g+|\nabla v|_g)\\
\quad+C(\|u\|_{L^\infty\Omega^1(M)}^{2}+\|v\|_{L^\infty\Omega^1(M)}^{2})|\nabla(u-v)|_g.
\end{multline*}
Using Sobolev embedding $W^{1,p}\Omega^1(M)\inclusion C\Omega^1(M)$, this implies that
$$
\|\,|(\Re\<u,\nabla\overline{u}\>_g u-\Re\<v,\nabla\overline{v}\>_g v)|_g\|_{L^p(M)}\le C(\|u\|_{W^{1,p}\Omega^1(M)}^{2}+\|v\|_{W^{1,p}\Omega^1(M)}^{2})\|u-v\|_{W^{1,p}\Omega^1(M)}.
$$
Combining all these three estimates with \eqref{eqn::w1p norm of fuu-fvv}, we finish the proof.
\end{proof}


\section{Asymptotics of the admittance map}\label{sec:asymp}
Let $(M,g)$ be a compact $3$-dimensional Riemannian manifold with smooth boundary. Suppose that $\varepsilon,\mu\in C^1(M)$ are complex functions with positive real parts and $a,b\in C^1(M)$. Let $m\ge 1$ be an integer and let $3<p\le 6$. Fix $\omega>0$ outside a discrete set of resonant frequencies. Suppose that $f\in TW^{1-1/p,p}_{\Div}(\p M)$ and $s\in\R$ is a small parameter. By Theorem~\ref{thm::main result 1}, there is a unique solution $(E^{(s)},H^{(s)})\in W^{1,p}_{\Div}(M)\times W^{1,p}_{\Div}(M)$ of \eqref{eqn::Maxwell} such that $\mathbf t(E^{(s)})=sf$ and
\begin{equation}\label{eqn::estimate for E^s and H^s}
\|E^{(s)}\|_{W^{1,p}_{\Div}(M)}+\|H^{(s)}\|_{W^{1,p}_{\Div}(M)}\le C|s|\,\|f\|_{TW^{1-1/p,p}_{\Div}(\p M)}.
\end{equation}
By Theorem~\ref{thm::well posedness new version homogeneous}, there is a unique $(E_1,H_1)\in W^{1,p}_{\Div}(M)\times W^{1,p}_{\Div}(M)$ solving \eqref{eqn::Maxwell homogeneous in appendix} with $\mathbf t(E_1)=f$ such that
\begin{equation}\label{eqn::estimate for E_1 and H_1}
\|E_1\|_{W^{1,p}_{\Div}(M)}+\|H_1\|_{W^{1,p}_{\Div}(M)}\le C\|f\|_{TW^{1-1/p,p}_{\Div}(\p M)}.
\end{equation}
Also, by Theorem~\ref{thm::well posedness new version} there is a unique solution $(E_2,H_2)\in W^{1,p}_{D}(M)\times W^{1,p}_{\Div}(M)$ for
$$
*dE_2=i\omega\mu H_2+i\omega b|H_1|_g^{2}H_1,\quad *dH_2=-i\omega \varepsilon E_2-i\omega a|E_1|_g^{2}E_1
$$
and satisfying
$$
\|E_2\|_{W^{1,p}_{\Div}(M)}+\|H_2\|_{W^{1,p}_{\Div}(M)}\le C\|\,|E_1|_g^{2}E_1\|_{W^{1,p}\Omega^1(M)}+C\|\,|H_1|_g^{2}H_1\|_{W^{1,p}\Omega^1(M)}.
$$
Then by Lemma~\ref{lemma::product estimate},
\begin{equation}\label{eqn::estimate for E_2 and H_2}
\begin{aligned}
\|E_2\|_{W^{1,p}_{\Div}(M)}+\|H_2\|_{W^{1,p}_{\Div}(M)}&\le C\|E_1\|^{3}_{W^{1,p}\Omega^1(M)}+C\|H_1\|^{3}_{W^{1,p}\Omega^1(M)}\\
&\le C\|E_1\|^{3}_{W^{1,p}_{\Div}(M)}+C\|H_1\|^{3}_{W^{1,p}_{\Div}(M)} \le C\|f\|_{TW^{1-1/p,p}_{\Div}(\p M)}^{3}.
\end{aligned}
\end{equation}
Now we define $(F^{(s)},G^{(s)})$ by
\begin{equation}\label{eqn::definition of F^s and G^s}
(E^{(s)},H^{(s)})=s(E_1+s^{2}F^{(s)},H_1+s^{2}G^{(s)}).
\end{equation}
Then by \eqref{eqn::estimate for E^s and H^s} and \eqref{eqn::estimate for E_1 and H_1}, $(F^{(s)},G^{(s)})$ satisfies
\begin{align*}
|s|^{3}\|F^{(s)}\|_{W^{1,p}_{\Div}(M)}&+|s|^{3}\|G^{(s)}\|_{W^{1,p}_{\Div}(M)}\\
&\le \|E^{(s)}\|_{W^{1,p}_{\Div}(M)}+\|H^{(s)}\|_{W^{1,p}_{\Div}(M)}+|s|\,\|E_1\|_{W^{1,p}_{\Div}(M)}+|s|\,\|H_1\|_{W^{1,p}_{\Div}(M)}\\
&\le C|s|\,\|f\|_{TW^{1-1/p,p}_{\Div}(\p M)}.
\end{align*}
Therefore,
\begin{equation}\label{eqn::estimate for F^s and G^s}
|s|^{2}\,\|F^{(s)}\|_{W^{1,p}_{\Div}(M)}+|s|^{2}\,\|G^{(s)}\|_{W^{1,p}_{\Div}(M)}\le C\|f\|_{TW^{1-1/p,p}_{\Div}(\p M)}.
\end{equation}

\begin{Lemma}\label{lemma::estimate for F^s-E_2 and G^s-H_2}
Suppose that $f\in TW^{1-1/p,p}_{\Div}(\p M)$. There is $s_0>0$ and there is $C_f>0$ depending on $f$, $\omega$ and $s_0$, but independent of $s$, such that for all $s\in \R$ with $|s|<s_0$,
$$
\|F^{(s)}-E_2\|_{W^{1,p}_{\Div}(M)}+\|G^{(s)}-H_2\|_{W^{1,p}_{\Div}(M)}\le C_f { |s|^2}.
$$
In particular,
\begin{equation}\label{eqn::estimate for F^s and G^s stronger}
\|F^{(s)}\|_{W^{1,p}_{\Div}(M)}+\|G^{(s)}\|_{W^{1,p}_{\Div}(M)}\le C_f.
\end{equation}
\end{Lemma}
\begin{proof}
Set $(P^{(s)},Q^{(s)})=(F^{(s)},G^{(s)})-(E_2,H_2)$. Then it is easy to see that
\begin{equation}\label{eqn::Maxwell in asymptotic expansion}
*dP^{(s)}=i\omega\mu Q^{(s)}+i\omega sh^{(s)},\quad
*dQ^{(s)}=-i\omega\varepsilon P^{(s)}-i\omega se^{(s)},\quad
\mathbf t(P^{(s)})=0,
\end{equation}
where
\begin{align*}
e^{(s)}&=s^{-1}{ a}\big(|E_1+s^{2}F^{(s)}|_g^{2}(E_1+s^{2}F^{(s)})-|E_1|^{2}_gE_1\big),\\
h^{(s)}&=s^{-1}{ b}\big(|H_1+s^{2}G^{(s)}|_g^{2}(H_1+s^{2}G^{(s)})-|H_1|^{2}_gH_1\big).
\end{align*}
Using  Lemma~\ref{lemma::product difference estimate},
\begin{align*}
\|e^{(s)}&\|_{W^{1,p}\Omega^1(M)}\le C{ |s|}\left[\|E_1\|_{W^{1,p}\Omega^1(M)}^{2}+(|s|^{2}\|F^{(s)}\|_{W^{1,p}\Omega^1(M)})^{2}\right]\|F^{(s)}\|_{W^{1,p}\Omega^1(M)},\\
\|h^{(s)}&\|_{W^{1,p}\Omega^1(M)}\le C{ |s|}\left[\|H_1\|_{W^{1,p}\Omega^1(M)}^{2}+(|s|^{2}\|G^{(s)}\|_{W^{1,p}\Omega^1(M)})^{2}\right]\|G^{(s)}\|_{W^{1,p}\Omega^1(M)}.
\end{align*}
Then by \eqref{eqn::estimate for E_1 and H_1} and \eqref{eqn::estimate for F^s and G^s},
\begin{align*}
\|e^{(s)}\|_{W^{1,p}\Omega^1(M)}+\|h^{(s)}&\|_{W^{1,p}\Omega^1(M)} \le C{ |s|}\|f\|_{TW^{1-1/p,p}_{\Div}(\p M)}^{2}(\|F^{(s)}\|_{W^{1,p}_{\Div}(M)}+\|G^{(s)}\|_{W^{1,p}_{\Div}(M)}).
\end{align*}
By Theorem~\ref{thm::well posedness new version}, $(P^{(s)},Q^{(s)})=(F^{(s)},G^{(s)})-(E_2,H_2)$ is the unique solution of \eqref{eqn::Maxwell in asymptotic expansion} and satisfies the estimate
\begin{align*}
\|F^{(s)}-E_2\|_{W^{1,p}_{\Div}(M)}+\|G^{(s)}-H_2\|_{W^{1,p}_{\Div}(M)}&=\|P^{(s)}\|_{W^{1,p}_{\Div}(M)}+\|Q^{(s)}\|_{W^{1,p}_{\Div}(M)}\\
&\le C\omega |s|\,(\|e^{(s)}\|_{W^{1,p}\Omega^1(M)}+\|h^{(s)}\|_{W^{1,p}\Omega^1(M)}).
\end{align*}
Therefore,
\begin{equation}\label{eqn::estimate for F^s-E_2 and G^s-H_2 intermediate}
\begin{aligned}
\|F^{(s)}-E_2\|&_{W^{1,p}_{\Div}(M)}+\|G^{(s)}-H_2\|_{W^{1,p}_{\Div}(M)}\\
&\le C\|f\|_{TW^{1-1/p,p}_{\Div}(\p M)}^{2}\omega { |s|^2}\,(\|F^{(s)}\|_{W^{1,p}_{\Div}(M)}+\|G^{(s)}\|_{W^{1,p}_{\Div}(M)}).
\end{aligned}
\end{equation}
Using the reverse triangle inequality to the left hand-side, we get
\begin{multline*}
\|F^{(s)}\|_{W^{1,p}_{\Div}(M)}+\|G^{(s)}\|_{W^{1,p}_{\Div}(M)}\\
\le C\|f\|_{TW^{1-1/p,p}_{\Div}(\p M)}^{2}\omega { |s|^2}\,(\|F^{(s)}\|_{W^{1,p}_{\Div}(M)}+\|G^{(s)}\|_{W^{1,p}_{\Div}(M)})\\
\qquad+\|E_2\|_{W^{1,p}_{\Div}(M)}+\|H_2\|_{W^{1,p}_{\Div}(M)}.
\end{multline*}
Using \eqref{eqn::estimate for E_2 and H_2}, this gives
\begin{align*}
\|&F^{(s)}\|_{W^{1,p}_{\Div}(M)}+\|G^{(s)}\|_{W^{1,p}_{\Div}(M)}\\
&\le C\|f\|_{TW^{1-1/p,p}_{\Div}(\p M)}^{2}\omega { |s|^2}\,(\|F^{(s)}\|_{W^{1,p}_{\Div}(M)}+\|G^{(s)}\|_{W^{1,p}_{\Div}(M)})+C\|f\|_{TW_{\Div}^{1-1/p,p}(\p M)}^{3}.
\end{align*}
The first term in the last line can be absorbed into the left hand-side by taking sufficiently small $s_0>0$ so that
$$
C\|f\|_{TW^{1-1/p,p}_{\Div}(\p M)}^{2}\omega { |s_0|^2}<1/2.
$$
Then we obtain
$$
\|F^{(s)}\|_{W^{1,p}_{\Div}(M)}+\|G^{(s)}\|_{W^{1,p}_{\Div}(M)}\le { C\|f\|_{TW_{\Div}^{1-1/p,p}(\p M)}^{3}}.
$$
Substituting this into \eqref{eqn::estimate for F^s-E_2 and G^s-H_2 intermediate}, we arrive to the desired estimate.
\end{proof}

{ Denote by $\Lambda^\omega_{\varepsilon,\mu}$ the admittance map $\Lambda^\omega_{\varepsilon,\mu,0,0}$ for linear Maxwell's equations.  We obtain the following asymptotic expansion of the admittance map.}
\begin{Proposition}\label{prop::asymptotics of the admittance map}
Suppose that $f\in TW^{1-1/p,p}_{\Div}(\p M)$ with $3<p\leq 6$. Then
\begin{align}
s^{-1}[\Lambda^\omega_{\varepsilon,\mu,a,b}(sf)-s\Lambda^\omega_{\varepsilon,\mu}(f)]\to 0\quad&\text{in}\quad W^{1-1/p,p}_{\Div}(\p M)\quad\text{as}\quad s\to 0,\label{eqn::1st term in asymptotics of the admittance map}\\
s^{-3}[\Lambda^\omega_{\varepsilon,\mu,a,b}(sf)-s\Lambda^\omega_{\varepsilon,\mu}(f)]\to \mathbf t(H_2)\quad&\text{in}\quad W^{1-1/p,p}_{\Div}(\p M)\quad\text{as}\quad s\to 0.\label{eqn::2nd term in asymptotics of the admittance map}
\end{align}
\end{Proposition}
\begin{proof}
From \eqref{eqn::definition of F^s and G^s} we have
$$
\Lambda^\omega_{\varepsilon,\mu,a,b}(sf)-s\Lambda^\omega_{\varepsilon,\mu}(f)=\mathbf t(H^{(s)})-s\mathbf t(H_1)=s^{3}\mathbf t(G^{(s)}).
$$
Then by boundedness of $\mathbf t$ from $W^{1,p}_{\Div}(M)$ onto $TW^{1-1/p,p}_{\Div}(\p M)$ and by \eqref{eqn::estimate for F^s and G^s stronger},
$$
\|s^{-1}[\Lambda^\omega_{\varepsilon,\mu,a,b}(sf)-s\Lambda^\omega_{\varepsilon,\mu}(f)]\|_{TW^{1-1/p,p}_{\Div}(\p M)}\le C|s|^{2}\,\|G^{(s)}\|_{W^{1,p}_{\Div}(M)}\le C_f|s|^{2}.
$$
Taking $s\to 0$, this implies \eqref{eqn::1st term in asymptotics of the admittance map}.\smallskip

Now, by boundedness of $\mathbf t$ from $W^{1,p}_{\Div}(M)$ onto $TW^{1-1/p,p}_{\Div}(\p M)$ and by Lemma~\ref{lemma::estimate for F^s-E_2 and G^s-H_2},
$$
\|s^{-3}[\Lambda^\omega_{\varepsilon,\mu,a,b}(sf)-s\Lambda^\omega_{\varepsilon,\mu}(f)]-\mathbf t(H_2)\|_{TW^{1-1/p,p}_{\Div}(\p M)}\le C\|G^{(s)}-H_2\|_{W^{1,p}_{\Div}(M)}\le C_f{ |s|^2}.
$$
Taking $s\to 0$, this implies \eqref{eqn::2nd term in asymptotics of the admittance map}.
\end{proof}

\section{Proof of Theorem~\ref{thm::main result 2}: Part I}\label{sec:mueps}
In this section we show that the material parameters and electric and magnetic susceptibilities of the nonlinear time-harmonic Maxwell equation \eqref{eqn::Maxwell} can be uniquely determined from the knowledge of admittance map.\smallskip

Let $(M,g)$ be a $3$-dimensional admissible manifold, that is $(M,g)\subset\subset \R\times (M_0,g_0)$ with $g=c(e\oplus g_0)$, where $c>0$ is a smooth function on $M$ and $(M_0,g_0)$ is a simple manifold of dimension two.\smallskip

The first ingredient in the proof of Theorem~\ref{thm::main result 2} is the reduction to the case $c=1$.
\begin{Lemma}
Let $(M,g)$ be a compact Riemannian $3$-dimensional manifold with boundary and let $c>0$ be a smooth function on $M$. Suppose that $\varepsilon,\mu\in C^\infty(M)$ with positive real parts and $a,b\in C^\infty(M)$. Then $\Lambda^\omega_{cg,\varepsilon,\mu,a,b}=\Lambda^\omega_{g,c^{1/2}\varepsilon,c^{1/2}\mu,c^{3/2}a,c^{3/2}b}$.
\end{Lemma}
\begin{proof}
Let $*_{cg}$ and $*_g$ denote the Hodge star operators corresponding to the metrics $cg$ and $g$, respectively. Following \cite[Lemma~7.1]{KenigSaloUhlmann2011a}, we note that { $*_{cg}u=c^{3/2-k}*_g u$} for a $k$-form $u$. Therefore, $(E,H)$ solves
$$
*_{cg}dE=i\omega\mu H+i\omega b|H|_{cg}^{2}H,\quad *_{cg}dH=-i\omega \varepsilon E-i\omega a|E|_{cg}^{2}E
$$
if and only if it solves
$$
*_{g}dE=i\omega c^{1/2}\mu H+i\omega c^{3/2}b|H|_{g}^{2}H,\quad *_{g}dH=-i\omega c^{1/2}\varepsilon E-i\omega c^{3/2}a|E|_{g}^{2}E.
$$
Therefore, $\Lambda^\omega_{cg,\varepsilon,\mu,a,b}=\Lambda^\omega_{g,c^{1/2}\varepsilon,c^{1/2}\mu,c^{3/2}a,c^{3/2}b}$.
\end{proof}

Therefore, it is enough to prove Theorem~\ref{thm::main result 2} in the case $c=1$. Thus, in the rest of this section we assume that $(M,g)\subset\subset \R\times (M_0,g_0)$ with $g=e\oplus g_0$, where $(M_0,g_0)$ is a simple manifold of dimension two.\smallskip

By \eqref{eqn::1st term in asymptotics of the admittance map} in Proposition~\ref{prop::asymptotics of the admittance map}, we obtain $\Lambda^\omega_{\varepsilon_1,\mu_1}=\Lambda^\omega_{\varepsilon_2,\mu_2}$. Then by \cite[Theorem~1.1]{KenigSaloUhlmann2011a}, we get $\varepsilon_1=\varepsilon_2$ and $\mu_1=\mu_2$ in $M$. In what follows, we write $\varepsilon=\varepsilon_1=\varepsilon_2$ and $\mu=\mu_1=\mu_2$.\smallskip

{\section{Construction of CGO solutions}\label{sec:CGO}

Our aim is to very briefly review the construction of CGO solutions; see \cite{KenigSaloUhlmann2011a} for details. In Section~\ref{sec:reduction to Hodge Schrodinger}, we recall the reduction of the Maxwell equations to the Hodge-Dirac and Schr\"odinger type equations, introduced in \cite{OlaSomersalo1996,KenigSaloUhlmann2011a}. Then, in Section~\ref{sec:CGO for Maxwell}, we restate the form of existence and basic properties of CGO solutions for Maxwell's equations using the reduction in Section~\ref{sec:reduction to Hodge Schrodinger}.

\subsection{Reduction to the Hodge-Schr\"odinger equation}\label{sec:reduction to Hodge Schrodinger} Let $(M,g)$ be a smooth compact Riemannian 3-dimensional manifold with boundary. The arguments require $\varepsilon,\mu\in C^2(M)$ to be complex functions with positive real parts. If $\Phi,\Psi$ are complex scalar functions on $M$ and $E,H$ are complex $1$-forms on $M$, we consider the graded forms $X=\Phi+E+*H+*\Psi$ and we denote them in vector notation
$$
X=(\begin{array}{cc|cc} \Phi&*H&*\Psi&E \end{array})^t.
$$
We define the following matrix operators acting on graded forms on $M$
\begin{align*}
P&=\frac{1}{i}(d-\delta)=\(\begin{array}{cc|cc}
&&&-\delta\\
&&-\delta&d\\ \hline
&d&&\\
d&-\delta&&
\end{array}\),\quad
V&=\(\begin{array}{cc|cc}
-\omega\mu&&&*(D\alpha\wedge*\,\cdot)\\
&-\omega\mu&*(D\alpha\wedge*\,\cdot)&\\ \hline
&D\beta\wedge&-\omega\varepsilon&\\
D\beta\wedge&&&-\omega\varepsilon
\end{array}\),
\end{align*}
where $D=-id$, $\alpha=\log\varepsilon$ and $\beta=\log\mu$. Note that $P$ is the self-adjoint Hodge-Dirac operator. It was shown in \cite[Section~3]{KenigSaloUhlmann2011a} that $(E,H)$ is a solution of the original Maxwell's equations
$$
*dE=i\omega\mu H,\quad *dH=-i\omega \varepsilon E
$$
if and only if $X=(\begin{array}{cc|cc} \Phi&*H&*\Psi&E \end{array})^t$ is a solution of $(P+V)X=0$ with $\Phi=\Psi=0$.
\smallskip

To reduce the Maxwell equations to the Schr\"odinger type equation, we consider the rescaling
$$
Y=\(\begin{array}{c|c}\mu^{1/2}\id_2&\\ \hline
&\varepsilon^{1/2}\id_2\end{array}\)X,
$$
where $\id_2$ is the $2\times 2$ identity matrix. We always assume that $X$ and $Y$ are related via this rescaling. We write the graded form $Y$ as
$$
Y=(\begin{array}{cc|cc} Y^0&Y^2&Y^3&Y^1 \end{array})^t,
$$
with $Y^k$ being the $k$-form part of $Y$. One can check that $(P+V)X=0$ if and only if $(P+W)Y=0$. Here
$$
W=-\kappa+\frac12 \(\begin{array}{cc|cc}
&&&*(D\alpha\wedge*\,\cdot)\\
&&*(D\alpha\wedge*\,\cdot)&-D\alpha\wedge\\ \hline
&D\beta\wedge&&\\
D\beta\wedge&*(D\beta\wedge*\,\cdot)&&
\end{array}\),\quad \kappa=\omega(\varepsilon\mu)^{1/2}.
$$
Then
$$
(P+W)(P-W^t)=-\Delta+Q,
$$
where $Q$ is $L^\infty$ potential. For the exact expression of $Q$, see \cite[Lemma~3.1]{KenigSaloUhlmann2011a}.

\subsection{CGO solutions for Maxwell's equations}\label{sec:CGO for Maxwell}

Let $(M,g)$ be a $3$-dimensional admissible manifold. Throughout this section, we assume that $M\subset \R\times M_0^{\rm int}$ and the metric $g$ has the form $g=e\oplus g_0$ and $(M_0,g_0)$ is simple. Choose another simple manifold $(\widetilde M_0,g_0)$ such that $M_0\subset\subset\widetilde M_0$ and choose $p\in \widetilde M_0\setminus M_0$. Simplicity of $(\widetilde M_0,g_0)$ implies that there are globally defined polar coordinates $(r,\theta)$ centered at $p$. In these coordinates, the metric $g$ has the form
\begin{equation}\label{eqn::form of g}
g=e\oplus \bigg(\,\begin{matrix}
1&0\\
0&m(r,\theta)
\end{matrix}\,\bigg),
\end{equation}
where $m$ is a smooth positive function.\smallskip

The following result states the existence and basic properties of CGO solutions.

\begin{Proposition}\label{prop::CGO solutions}
Let $(M,g)$ be a $3$-dimensional admissible manifold with $g=e\oplus g_0$ and let $2\le p\le 6$. Suppose that $\varepsilon,\mu\in C^3(M)$ with $\Re(\varepsilon),\Re(\mu)>0$ in $M$. Let $s_0,t_0\in\R$ and $\lambda\in\R\setminus\{0\}$ be constants and let $\chi\in C^\infty(S^1)$. Then for $\tau\in\R$ with sufficiently large $|\tau|>0$ and outside a countable subset of $\R$, the Maxwell's equations
\begin{equation}\label{eqn::Maxwell equation in CGO section}
*dE=i\omega\mu H,\quad *dH=-i\omega \varepsilon E
\end{equation}
has a solution $(E,H)\in W^{1,p}_{\Div}(M)\times W^{1,p}_{\Div}(M)$ of the form
\begin{align*}
E&=e^{-\tau(x_1+ir)}\Big[t_0\varepsilon^{-1/2}|g|^{-1/4}e^{i\lambda(x_1+ir)}\chi(\theta)(dx_1+idr)+R\Big],\\
H&=e^{-\tau(x_1+ir)}\Big[s_0\mu^{-1/2}|g|^{-1/4}e^{i\lambda(x_1+ir)}\chi(\theta)(dx_1+idr)+R'\Big],
\end{align*}
where $R,R'\in W^{1,p}_{\Div}(M)$ are correction terms satisfying the estimates
\begin{equation}\label{est::Lp estimates for R and R'}
\|R\|_{L^p\Omega^1(M)},\|R'\|_{L^p\Omega^1(M)}\le C\frac{1}{|\tau|^{\frac{6-p}{2p}}},
\end{equation}
with $C>0$ constant independent of $\tau$. 
Note that when $p< 6$ the remainders decay as $\tau$ increases. 

\end{Proposition}

\begin{Remark}{\rm The proof follows by carefully rewriting the result \cite[Theorem~3.1]{ChungOlaSaloTzou2016} (see also \cite[Theorem~6.1(a)]{KenigSaloUhlmann2011a} for the original result) which requires $\varepsilon,\mu\in C^3(M)$. We need the former result in order to get \eqref{est::Lp estimates for R and R'} for all $2\ge p\le 6$ rather than just $p=2$. This will be very important in the next section.}
\end{Remark}

\begin{proof}
By \cite[Theorem~3.1]{ChungOlaSaloTzou2016}, for $\tau\in\R$ with sufficiently large $|\tau|>0$ and outside a countable subset of $\R$, there is a solution for $(-\Delta+Q)Z=0$ such that $Z\in H^3\Omega(M)$ and
$$
(P+W)Y=0,\quad Y=(P-W^t)Z,\quad Y^0=Y^3=0\quad\text{in}\quad M,
$$
and having the form
$$
Z=e^{-\tau(x_1+ir)}(A+R_0),\quad A=-i|g|^{-1/4}e^{i\lambda(x_1+ir)}\chi(\theta) (\begin{array}{cc|cc} s_0&0&t_0*1&0 \end{array})^t
$$
and $\|R_0\|_{H^s\Omega(M)}\le C|\tau|^{1-s}$, $0\le s\le 2$, where $C>0$ is a constant independent of $\tau$.\smallskip

Let us compute $Y^1$ and $Y^3$. Writing $\rho=x_1+ir$ and using the fact that $A^1=A^2=0$, one can see that
$$
Y^1=[(P-W^t)Z]^1=e^{-\tau\rho}(y_1+r_1),\quad Y^2=[(P-W^t)Z]^2=e^{-\tau\rho}(y_2+r_2),
$$
where
\begin{align*}
y_1&=-\frac{\tau}{i}i_{d\rho}A^2,\quad r_1=-[W^t(A+R_0)]^1-\frac{1}{i}\delta A^2+\frac{1}{i}dR^0_0-\frac{\tau}{i}R_0^0\,d\rho-\frac{1}{i}\delta R_0^2-\frac{\tau}{i}i_{d\rho}R_0^2,\\
y_2&=-\frac{\tau}{i}i_{d\rho}A^3,\quad r_2=-[W^t(A+R_0)]^2-\frac{1}{i}\delta A^3+\frac{1}{i}dR^1_0-\frac{\tau}{i}d\rho\wedge R_0^1-\frac{1}{i}\delta R_0^3-\frac{\tau}{i}i_{d\rho}R_0^3.
\end{align*}
It is easy to see that
$$
\|r_1\|_{H^s\Omega^1(M)},\|r_2\|_{H^s\Omega^2(M)}\le C|\tau|^{s-1},\quad 0\le s\le 1.
$$
Using \eqref{eqn::expression for contraction}, one can show that
$$
y_1=s_0\tau |g|^{-1/4}e^{i\lambda\rho}\chi(\theta)\,d\rho,\quad y_2=t_0\tau |g|^{-1/4}e^{i\lambda\rho}\chi(\theta)*\,d\rho.
$$
Note that $Y_0:=\tau^{-1} Y$ will solve $(P+W)Y_0=0$ with $Y_0^0=Y_0^3=0$. If we define
\begin{align*}
E:&=\varepsilon^{-1/2}Y_0^1=e^{-\tau(x_1+ir)}\Big[s_0\varepsilon^{-1/2}|g|^{-1/4}e^{i\lambda(x_1+ir)}\chi(\theta)(dx_1+idr)+\underbrace{\varepsilon^{-1/2}r_1}_{R}\Big],\\
H:&=\mu^{-1/2}*Y_0^2=e^{-\tau(x_1+ir)}\Big[t_0\mu^{-1/2}|g|^{-1/4}e^{i\lambda(x_1+ir)}\chi(\theta)(dx_1+idr)+\underbrace{\mu^{-1/2}r_2}_{R'}\Big].
\end{align*}
Then $(E,H)\in H^2\Omega^1(M)\times H^2\Omega^1(M)$ will be a solution of the Maxwell's equations \eqref{eqn::Maxwell equation in CGO section} and the correction terms $R,R'$ satisfy the estimates
\begin{equation}\label{est::Hs estimates for R and R'}
\|R\|_{H^s\Omega^1(M)},\|R'\|_{H^s\Omega^1(M)}\le C|\tau|^{s-1},\quad 0\le s\le 1,
\end{equation}
with $C>0$ constant independent of $\tau$.\smallskip

By Sobolev embedding, we have $(E,H)\in W^{1,p}\Omega^1(M)\times W^{1,p}\Omega^1(M)$ for $2\le p\le 6$. Then, using \eqref{eqn::Maxwell equation in CGO section} and \eqref{eqn::expression for Div}, it is straightforward to check that $\mathbf t(E),\mathbf t(H)\in TW^{1-1/p,p}_{\Div}(\p M)$. Thus, $(E,H)\in W^{1,p}_{\Div}(M)\times W^{1,p}_{\Div}(M)$ for $2\le p\le 6$.\smallskip

Finally, one can obtain the estimates in \eqref{est::Lp estimates for R and R'}, using the inequality $0\le \frac{3p-6}{2p}\le 1$, Sobolev embedding and \eqref{est::Hs estimates for R and R'}.
\end{proof}

}\section{Proof of Theorem~\ref{thm::main result 2}: Part II}\label{sec:ab}

In this section we continue proof of Theorem~\ref{thm::main result 2}. Our aim is to show that $a_1=a_2$ and $b_1=b_2$. To that end, we shall use complex geometrical optics solutions, constructed in the previous section, in the following integral identity~\eqref{eqn::main integral identity after second polarization}.\smallskip

\subsection{An important energy integral identity}
{ Now, by \eqref{eqn::2nd term in asymptotics of the admittance map} in Proposition~\ref{prop::asymptotics of the admittance map}, we obtain $\mathbf t(H_2^{1})=\mathbf t(H_2^2)$, where $(E_2^j,H_2^j)\in W^{1,p}_{D}(M)\times W^{1,p}_{\Div}(M)$, $j=1,2$, is the unique solution of
\begin{equation}\label{eqn::nonlinear Maxwell equation with j-indice}
*dE_2^j=i\omega\mu H_2^j+i\omega b_j |H_1|_g^{2}H_1,\quad
*dH_2^j=-i\omega \varepsilon E_2^j-i\omega a_j |E_1|_g^{2}E_1
\end{equation}
with $\mathbf t(E_2^j)=0$ and $(E_1,H_1)\in W^{1,p}_{\Div}(M)\times W^{1,p}_{\Div}(M)$ is a solution of
\begin{equation}\label{eqn::linear Maxwell equation with epsilon and mu}
* dE_1=i\omega\mu H_1,\quad
* dH_1=-i\omega \varepsilon E_1
\end{equation}
satisfying $\mathbf t(E_1)=f$. Let $(E,H) \in W^{1,p}_{\Div}(M)\times W^{1,p}_{\Div}(M)$ be a solution of
\begin{equation}\label{eqn::linear Maxwell equation with bar-epsilon and bar-mu}
* dE=i\omega\overline \mu H,\quad
* dH=-i\omega\overline \varepsilon E.
\end{equation}
Using integration by parts,
\begin{align*}
(\mathbf t(H_2^j)|\mathbf t(i_\nu*E))_{L^2\Omega^1(\p M)}&=(dH_2^j|*E)_{L^2\Omega^2(M)}-(H_2^j|\delta(*E))_{L^2\Omega^1(M)}\\
&=(*dH_2^j|E)_{L^2\Omega^1(M)}-(H_2^j|*dE)_{L^2\Omega^1(M)}.
\end{align*}
Since $(E_2^j,H_2^j)$ satisfy \eqref{eqn::nonlinear Maxwell equation with j-indice} and $(E,H)$ satisfy \eqref{eqn::linear Maxwell equation with bar-epsilon and bar-mu}, we can show
\begin{align*}
(\mathbf t(H_2^j)|\mathbf t(i_\nu*E))_{L^2\Omega^1(\p M)}&=-(i\omega \varepsilon E_2^j|E)_{L^2\Omega^1(M)}-(i\omega a_j |E_1|_g^{2}E_1|E)_{L^2\Omega^1(M)}-(H_2^j|i\omega\overline \mu H)_{L^2\Omega^1(M)}\\
&=(E_2^j|i\omega\overline \varepsilon E)_{L^2\Omega^1(M)}-(i\omega a_j |E_1|_g^{2}E_1|E)_{L^2\Omega^1(M)}+(i\omega \mu H_2^j|H)_{L^2\Omega^1(M)}.
\end{align*}
{Here and in what follows, all integrals make sense because of the assumption $p\ge 4$.} Since $(E_2^j,H_2^j)$ satisfy \eqref{eqn::nonlinear Maxwell equation with j-indice} and $(E,H)$ satisfy \eqref{eqn::linear Maxwell equation with bar-epsilon and bar-mu}, this can be rewritten as
\begin{align*}
(\mathbf t(H_2^j)|\mathbf t(i_\nu*E))_{L^2\Omega^1(\p M)}&=(*dE_2^j|H)_{L^2\Omega^1(M)}-(E_2^j|*dH)_{L^2\Omega^1(M)}\\
&\quad-(i\omega a_j |E_1|_g^{2}E_1|E)_{L^2\Omega^1(M)}-(i\omega b_j|H_1|_g^{2}H_1|H)_{L^2\Omega^1(M)}\\
&=(dE_2^j|*H)_{L^2\Omega^2(M)}-(E_2^j|\delta(*H))_{L^2\Omega^1(M)}\\
&\quad-(i\omega a_j |E_1|_g^{2}E_1|E)_{L^2\Omega^1(M)}-(i\omega b_j|H_1|_g^{2}H_1|H)_{L^2\Omega^1(M)}.
\end{align*}
Using integration by parts and the fact that $\mathbf t(E_2^j)=0$, we can show that
$$
(dE_2^j|*H)_{L^2\Omega^2(M)}-(E_2^j|\delta(*H))_{L^2\Omega^1(M)}=0.
$$
Therefore,
$$
-\frac{1}{i\omega}(\mathbf t(H_2^j)|\mathbf t(i_\nu*E))_{L^2\Omega^1(\p M)}=(a_j |E_1|_g^{2}E_1|E)_{L^2\Omega^1(M)}+(b_j|H_1|_g^{2}H_1|H)_{L^2\Omega^1(M)}.
$$
Since $\mathbf t(H_2^{1})=\mathbf t(H_2^2)$, this implies that
\begin{equation}\label{eqn::main integral identity}
((a_1-a_2)|E_1|_g^{2}E_1|E)_{L^2\Omega^1(M)}+((b_1-b_2)|H_1|_g^{2}H_1|H)_{L^2\Omega^1(M)}=0
\end{equation}
for all $(E_1,H_1)\in W^{1,p}_{\Div}(M)\times W^{1,p}_{\Div}(M)$ solving \eqref{eqn::linear Maxwell equation with epsilon and mu} and for all $(E,H) \in W^{1,p}_{\Div}(M)\times W^{1,p}_{\Div}(M)$ solving \eqref{eqn::linear Maxwell equation with bar-epsilon and bar-mu}.
\smallskip

Note that if $(\widetilde E,\widetilde H),(E',H')\in W^{1,p}_{\Div}(M)\times W^{1,p}_{\Div}(M)$ solve \eqref{eqn::linear Maxwell equation with epsilon and mu}, then $(\widetilde E+E',\widetilde H+H')$ also solves \eqref{eqn::linear Maxwell equation with epsilon and mu}. Therefore, polarizing \eqref{eqn::main integral identity} by setting $(E_1,H_1)=(\widetilde E+E',\widetilde H+H')$, we obtain
\begin{equation}\label{eqn::main integral identity after first polarization}
\begin{aligned}
0&=\big((a_1-a_2)\Big[|\widetilde E|_g^{2}E'+2\Re \<E',\overline{\widetilde E}\>_g E'+|E'|^2_g\widetilde E+2\Re \<E',\overline{\widetilde E}\>_g \widetilde E\Big]\big|E\big)_{L^2\Omega^1(M)}\\
&\quad+\big((b_1-b_2)\Big[|\widetilde H|_g^{2}H'+2\Re \<H',\overline{\widetilde H}\>_g H'+|H'|^2_g\widetilde H+2\Re \<H',\overline{\widetilde H}\>_g \widetilde H\Big]\big|H\big)_{L^2\Omega^1(M)}
\end{aligned}
\end{equation}
for all $(\widetilde E,\widetilde H),(E',H')\in W^{1,p}_{\Div}(M)\times W^{1,p}_{\Div}(M)$ solving \eqref{eqn::linear Maxwell equation with epsilon and mu}.
\smallskip

Now, take $(E_{(j)},H_{(j)})\in  W^{1,p}_{\Div}(M)\times W^{1,p}_{\Div}(M)$, $j=1,2,3$, which solve \eqref{eqn::linear Maxwell equation with epsilon and mu}. Setting $(E_1,H_1)=(E_{(1)}+E_{(2)}+E_{(3)},H_{(1)}+H_{(2)}+H_{(3)})$ in \eqref{eqn::main integral identity} and using \eqref{eqn::main integral identity after first polarization}, we get
\begin{equation}\label{eqn::main integral identity after second polarization}
\begin{aligned}
0&=\big((a_1-a_2) \Re \<E_{(3)},\overline{E}_{(2)}\>_g E_{(1)}|E\big)_{L^2\Omega^1(M)}+\big((a_1-a_2) \Re \<E_{(3)},\overline{E}_{(1)}\>_g E_{(2)}|E\big)_{L^2\Omega^1(M)}\\
&\qquad+\big((a_1-a_2)\Re \<E_{(1)},\overline{E}_{(2)}\>_g E_{(3)}|E\big)_{L^2\Omega^1(M)}\\
&+\big((b_1-b_2) \Re \<H_{(3)},\overline{H}_{(2)}\>_g H_{(1)}|H\big)_{L^2\Omega^1(M)}+\big((b_1-b_2) \Re \<H_{(3)},\overline{H}_{(1)}\>_g H_{(2)}|H\big)_{L^2\Omega^1(M)}\\
&\qquad+\big((b_1-b_2) \Re \<H_{(1)},\overline{H}_{(2)}\>_g H_{(3)}|H\big)_{L^2\Omega^1(M)}
\end{aligned}
\end{equation}
for all $(E_{(j)},H_{(j)})\in  W^{1,p}_{\Div}(M)\times W^{1,p}_{\Div}(M)$, $j=1,2,3$, solving \eqref{eqn::linear Maxwell equation with epsilon and mu}.\\
}

\subsection{Proof of Theorem~\ref{thm::main result 2}: Part II}{

Recall that we assume $M\subset \R\times M_0^{\rm int}$ and the metric has the form $g=e\oplus g_0$, where $e$ is Euclidean metric on $\R$ and $(M_0,g_0)$ is a simple $2$-dimensional manifold. \smallskip

Recall that we assume $3< p<6$. Using Proposition~\ref{prop::CGO solutions}, for $\tau\in\R$ with sufficiently large $|\tau|$, for arbitrary ${ \chi}\in C^\infty(S^1)$, $s_0,t_0\in\R$ and $\lambda\in\R\setminus \{0\}$, there are $(E_{(j)},H_{(j)})\in  W^{1,p}_{\Div}(M)\times W^{1,p}_{\Div}(M)$, $j=1,2,3$, solving \eqref{eqn::linear Maxwell equation with epsilon and mu} and there is $(E,H)\in  W^{1,p}_{\Div}(M)\times W^{1,p}_{\Div}(M)$ solving \eqref{eqn::linear Maxwell equation with bar-epsilon and bar-mu} of the forms
\begin{align*}
E_{(1)}&=e^{-\tau(x_1+ir)}\Big[t_0\,\varepsilon^{-1/2}|g|^{-1/4} e^{i\lambda(x_1+ir)}\chi(\theta)(dx_1+idr)+R_{(1)}\Big]=e^{-\tau(x_1+ir)}(A_{(1)}+R_{(1)}),\\
H_{(1)}&=e^{-\tau(x_1+ir)}\Big[s_0\,\mu^{-1/2}|g|^{-1/4} e^{i\lambda(x_1+ir)}\chi(\theta)(dx_1+idr)+R_{(1)}'\Big]=e^{-\tau(x_1+ir)}(A'_{(1)}+R_{(1)}'),\\
E_{(2)}&=e^{\tau(x_1-ir)}\Big[\varepsilon^{-1/2}|g|^{-1/4} e^{i\lambda(x_1-ir)}(dx_1-idr)+R_{(2)}\Big]=e^{\tau(x_1-ir)}(A_{(2)}+R_{(2)}),\\
H_{(2)}&=e^{\tau(x_1-ir)}\Big[ \mu^{-1/2}|g|^{-1/4} e^{i\lambda(x_1-ir)}(dx_1-idr)+R_{(2)}'\Big]=e^{\tau(x_1-ir)}(A'_{(2)}+R_{(2)}'),\\
E_{(3)}&=e^{-\tau(x_1-ir)}\Big[\varepsilon^{-1/2}|g|^{-1/4} e^{-i\lambda(x_1-ir)}(dx_1-idr)+R_{(3)}\Big]=e^{-\tau(x_1-ir)}(A_{(3)}+R_{(3)}),\\
H_{(3)}&=e^{-\tau(x_1-ir)}\Big[ \mu^{-1/2}|g|^{-1/4} e^{-i\lambda(x_1-ir)}(dx_1-idr)+R_{(2)}'\Big]=e^{-\tau(x_1-ir)}(A'_{(3)}+R_{(3)}'),\\
E&=e^{\tau(x_1+ir)}\Big[\varepsilon^{-1/2}|g|^{-1/4} e^{i\lambda(x_1+ir)}(dx_1+idr)+R\Big]=e^{-\tau(x_1+ir)}(A+R),\\
H&=e^{\tau(x_1+ir)}\Big[\mu^{-1/2}|g|^{-1/4} e^{i\lambda(x_1+ir)}(dx_1+idr)+R'\Big]=e^{-\tau(x_1+ir)}(A'+R'),
\end{align*}
with
$$
\|R_{(j)}\|_{L^p\Omega^1(M)},\|R'_{(j)}\|_{L^p\Omega^1(M)},\|R\|_{L^p\Omega^1(M)},\|R'\|_{L^p\Omega^1(M)}\le C\frac{1}{|\tau|^{\frac{6-p}{2p}}},\quad j=1,2,3.
$$
Since we assume $p<6$, these imply
\begin{equation}\label{eqn::decay for correction terms}
\|R_{(j)}\|_{L^p\Omega^1(M)},\|R'_{(j)}\|_{L^p\Omega^1(M)},\|R\|_{L^p\Omega^1(M)},\|R'\|_{L^p\Omega^1(M)}\le o(1)\quad\text{as }\tau\to\infty,\quad j=1,2,3.
\end{equation}
Substituting these solutions into~\eqref{eqn::main integral identity after second polarization}, we get
\begin{equation}\label{eqn::integral identity second step}
\begin{aligned}
0&=\big((a_1-a_2) \Re\<A_{(3)}+R_{(3)},\overline A_{(2)}+\overline R_{(2)}\>_g(A_{(1)}+R_{(1)})|(A+R)\big)_{L^2\Omega^1(M)}\\
&\quad +\big((a_1-a_2) \Re\<A_{(3)}+R_{(3)},\overline A_{(1)}+\overline R_{(1)}\>_g(A_{(2)}+R_{(2)})|(A+R)\big)_{L^2\Omega^1(M)}\\
&\quad +\big((a_1-a_2) \Re\<A_{(1)}+R_{(1)},\overline A_{(2)}+\overline R_{(2)}\>_g(A_{(3)}+R_{(3)})|(A+R)\big)_{L^2\Omega^1(M)}\\
&\quad+\big((b_1-b_2) \Re\<A'_{(3)}+R'_{(3)},\overline A'_{(2)}+\overline R'_{(2)}\>_g(A'_{(1)}+R'_{(1)})|(A'+R')\big)_{L^2\Omega^1(M)}\\
&\quad +\big((b_1-b_2) \Re\<A'_{(3)}+R'_{(3)},\overline A'_{(1)}+\overline R'_{(1)}\>_g(A'_{(2)}+R'_{(2)})|(A'+R')\big)_{L^2\Omega^1(M)}\\
&\quad +\big((b_1-b_2) \Re\<A'_{(1)}+R'_{(1)},\overline A'_{(2)}+\overline R'_{(2)}\>_g(A'_{(3)}+R'_{(3)})|(A'+R')\big)_{L^2\Omega^1(M)}.
\end{aligned}
\end{equation}
Letting $\tau\to\infty$, we come to
\begin{equation}\label{eqn::integral identity third step}
\begin{aligned}
0&=\big((a_1-a_2) \Re\<A_{(3)},\overline A_{(2)}\>_gA_{(1)}|A\big)_{L^2\Omega^1(M)}+\big((a_1-a_2) \Re\<A_{(3)},\overline A_{(1)}\>_gA_{(2)}|A\big)_{L^2\Omega^1(M)}\\
&\qquad +\big((a_1-a_2) \Re\<A_{(1)},\overline A_{(2)}\>_gA_{(3)}|A\big)_{L^2\Omega^1(M)}\\
&\quad+\big((b_1-b_2) \Re\<A'_{(3)},\overline A'_{(2)}\>_gA'_{(1)}|A'\big)_{L^2\Omega^1(M)}+\big((b_1-b_2) \Re\<A'_{(3)},\overline A'_{(1)}\>_gA'_{(2)}|A'\big)_{L^2\Omega^1(M)}\\
&\qquad +\big((b_1-b_2) \Re\<A'_{(1)},\overline A'_{(2)}\>_gA'_{(3)}|A'\big)_{L^2\Omega^1(M)}.
\end{aligned}
\end{equation}
To see this, one expands every term in \eqref{eqn::integral identity second step} and uses generalized H\"older's inequality together with \eqref{eqn::decay for correction terms}. Then all terms in \eqref{eqn::integral identity second step} go to zero as $\tau\to\infty$ except the terms written in \eqref{eqn::integral identity third step}.\smallskip

Recall that the amplitudes $A_{(j)}$, $A'_{(j)}$, $j=1,2,3$, and $A$, $A'$ are of the form
\begin{align*}
A_{(1)}&=t_0\,\varepsilon^{-1/2}|g|^{-1/4} e^{i\lambda(x_1+ir)}\chi(\theta)(dx_1+idr),\quad &A'_{(1)}&=s_0\,\mu^{-1/2}|g|^{-1/4} e^{i\lambda(x_1+ir)}\chi(\theta)(dx_1+idr),\\
A_{(2)}&=\varepsilon^{-1/2}|g|^{-1/4} e^{i\lambda(x_1-ir)}(dx_1-idr),\quad &A'_{(2)}&=\mu^{-1/2}|g|^{-1/4} e^{i\lambda(x_1-ir)}(dx_1-idr),\\
A_{(3)}&=\varepsilon^{-1/2}|g|^{-1/4} e^{-i\lambda(x_1-ir)}(dx_1-idr),\quad &A'_{(3)}&=\mu^{-1/2}|g|^{-1/4} e^{-i\lambda(x_1-ir)}(dx_1-idr),\\
A&=\varepsilon^{-1/2}|g|^{-1/4} e^{i\lambda(x_1+ir)}(dx_1+idr),\quad &A'&=\mu^{-1/2}|g|^{-1/4} e^{i\lambda(x_1+ir)}(dx_1+idr).
\end{align*}
Then the second, third, fifth and last terms in \eqref{eqn::integral identity third step} will vanish. Considering the cases $t_0=1$, $s_0=0$ and $t_0=0$, $s_0=1$ separately, we obtain
$$
\int_M fe^{-i2\lambda(x_1-ir)}\chi(\theta)|g|^{-1/2}\,d\Vol_g=0\quad\text{and}\quad\int_M he^{-i2\lambda(x_1-ir)}\chi(\theta)|g|^{-1/2}\,d\Vol_g=0,
$$
where
$$
f:=\frac{a_1-a_2}{|\varepsilon|^2|g|^{1/2}}\quad\text{and}\quad h:=\frac{b_1-b_2}{|\mu|^2|g|^{1/2}}.
$$
Now, we extend $f$ and $h$ as zero to $\R\times M_0$. Since $d\Vol_g=|g|^{1/2}\,dx_1drd\theta$, we get
$$
\int_{S^1}\chi(\theta)\int_0^\infty e^{-2\lambda r}\Big(\int_{-\infty}^\infty f e^{-i2\lambda x_1}\,dx_1\Big)\,drd\theta=0
$$
and
$$
\int_{S^1}\chi(\theta)\int_0^\infty e^{-2\lambda r}\Big(\int_{-\infty}^\infty h e^{-i2\lambda x_1}\,dx_1\Big)\,drd\theta=0.
$$
Varying $\chi\in C^\infty(S^1)$ and noting that the terms in the brackets are the one-dimensional Fourier transforms of $f$ and $h$ with respect to the $x_1$-variable, which we denote by $\widehat f$ and $\widehat h$, respectively, we get
$$
\int_0^\infty e^{-2\lambda r} \widehat f\,(2\lambda,r,\theta)\,dr=\int_0^\infty e^{-2\lambda r} \widehat h\,(2\lambda,r,\theta)\,dr=0,\quad \theta\in S^1.
$$
Recall that $(r,\theta)$ are polar coordinates in $M_0$. Therefore, $r\mapsto(r,\theta)$ is a geodesic in $M_0$ and the integrals above are the attenuated geodesic ray transforms of $\widehat f$ and $\widehat h$ on $M_0$ with constant attenuation $-2\lambda$. Then injectivity of this transform on simple manifolds of dimension two \cite[Theorem~1.1]{SaloUhlmann2011} implies that $\widehat f\,(2\lambda,\cdot)=\widehat h\,(2\lambda,\cdot)=0$ in $M_0$ for all $\lambda\in\R\setminus \{0\}$. Now, using the uniqueness result for the Fourier transform, we show that $f=h=0$ and hence $a_1=a_2$ and $b_1=b_2$ in $M$, finishing the proof of Theorem~\ref{thm::main result 2}.

}

\appendix

\section{Regularity of solutions of linear Maxwell equations}
In this section we prove the following regularity result for linear, inhomogeneous, time-harmonic Maxwell equations.
\begin{Theorem}\label{thm::main regularity result}
Let $2\leq p\leq 6$ and let $\varepsilon,\mu\in C^1(M)$ be complex functions with positive real parts. Suppose that $J_e,J_m\in W^p_\delta\Omega^1(M)$ and $i_\nu J_e|_{\p M}, i_\nu J_m|_{\p M}\in W^{1-1/p,p}(\p M)$. If $(E,H)\in W^{p}_d\Omega^1(M)\times W^{p}_d\Omega^1(M)$ is a solution of the Maxwell's system
\begin{equation}\label{eqn::Maxwell in Appendix1}
\begin{cases}
* dE=i\omega\mu H+J_m,\\
* dH=-i\omega \varepsilon E+J_e
\end{cases}
\end{equation}
such that $\mathbf t(E)\in TW^{1-1/p,p}_{\Div}(\p M)$, then $(E,H)\in W^{1,p}_{\Div}(M)\times W^{1,p}_{\Div}(M)$ and 
\begin{align*}
\|E\|_{W^{1,p}_{\Div}(M)}+\|H&\|_{W^{1,p}_{\Div}(M)}\\
&\le C(\|\mathbf t(E)\|_{TW^{1-1/p,p}_{\Div}(\p M)}+\|J_e\|_{W^{p}_\delta\Omega^1(M)}+\|J_m\|_{W^{p}_\delta\Omega^1(M)})\\
&\qquad+C(\|i_\nu J_e|_{\p M}\|_{W^{1-1/p,p}(\p M)}+\|i_\nu J_m|_{\p M}\|_{W^{1-1/p,p}(\p M)})
\end{align*}
for some constant $C>0$ independent of $E$, $H$, $J_e$ and $J_m$.
\end{Theorem}

\begin{proof}
We apply $\delta$ to the Maxwell equations \eqref{eqn::Maxwell in Appendix1} to obtain
\begin{equation}\label{eqn::divergence nonhomogenous Maxwell in appendix}
\begin{cases}
\delta H=i_{d\log\mu}H-(i\omega\mu)^{-1}\delta J_m,\\
\delta E=i_{d\log\varepsilon}E+(i\omega\varepsilon)^{-1}\delta J_e.
\end{cases}
\end{equation}
Since $E,H\in W^{p}_d\Omega^1(M)$ and $J_e,J_m\in W^p_\delta\Omega^1(M)$, this clearly implies that $E,H\in W^{p}_d\Omega^1(M)\cap W^p_{\delta}\Omega^1(M)$. Now, by Proposition~\ref{prop::bounded embedding of W^p_{d,delta} with regular tangential trace into W^1p}, we get $E\in W^{1,p}_{\Div}(M)$ and
$$
\|E\|_{W^{1,p}_{\Div}(M)}\le C(\|E\|_{W^{p}_d\Omega^1(M)}+\|\mathbf t(E)\|_{TW^{1-1/p,p}_{\Div}(\p M)}+\|J_e\|_{W^{p}_\delta\Omega^1(M)}+\|J_m\|_{W^{p}_\delta\Omega^1(M)}),
$$
since $\mathbf t(E)\in TW^{1-1/p,p}_{\Div}(\p M)$.\smallskip

To show that $H\in W^{1,p}\Omega^1(M)$, we use similar reasonings. Using \eqref{eqn::Maxwell in Appendix1} and \eqref{eqn::expression for Div} we get
$$
i_\nu H|_{\p M}=\frac{1}{i\omega \mu}i_\nu*dE|_{\p M}-\frac{1}{i\omega \mu}i_\nu J_m|_{\p M} \in W^{1-1/p,p}(\p M),
$$
since $\mathbf t(E)\in TW^{1-1/p,p}_{\Div}(\p M)$ and $i_\nu J_m|_{\p M}\in W^{1-1/p,p}(\p M)$. Then by Proposition~\ref{prop::bounded embedding of W^p_{d,delta} with regular normal trace into W^1p}, we have $H\in W^{1,p}\Omega^1(M)$ and
\begin{align*}
\|H\|_{W^{1,p}\Omega^1(M)}&\le C(\|H\|_{W^p_d\Omega^1(M)}+\|\mathbf t(E)\|_{TW^{1-1/p,p}_{\Div}(\p M)})\\
&\qquad+C(\|J_m\|_{W^p_\delta\Omega^1(M)}+\|i_\nu J_m|_{\p M}\|_{W^{1-1/p,p}(\p M)}).
\end{align*}
We also have $H\in W^{1,p}_{\Div}(M)$, since $\mathbf t(H)\in W^{1-1/p,p}(\p M)$ and
$$
\Div(\mathbf t(H))=i_\nu *dH|_{\p M}=-i\omega \varepsilon\, i_\nu E|_{\p M}+i_\nu J_e|_{\p M}\in W^{1-1/p,p}(\p M),
$$
where we have used \eqref{eqn::expression for Div} and the hypothesis $i_\nu J_e|_{\p M}\in W^{1-1/p,p}(\p M)$. Finally, the estimate in the statement of the theorem follows by combining all the above estimates. The proof is complete.
\end{proof}

\end{document}